\documentclass[journal]{IEEEtran}

\usepackage{cite}

\usepackage{float}
\usepackage[caption=false,font=footnotesize]{subfig}
\usepackage{stfloats}

\usepackage{graphicx}
\graphicspath{{Figures/}}
\usepackage{xcolor}

\usepackage{algorithm} 
\usepackage{algorithmicx,algpseudocode}

\usepackage{amsmath,amsfonts,amssymb}
\usepackage{mathtools}

\DeclarePairedDelimiterX{\Paren}[1]{(}{)}{#1}
\DeclarePairedDelimiterX{\Brace}[1]{\{}{\}}{#1}
\DeclarePairedDelimiterX{\Brack}[1]{[}{]}{#1}
\DeclarePairedDelimiterX{\Abs}[1]{\rvert}{\lvert}{#1}
\DeclarePairedDelimiterX{\Norm}[1]{\lVert}{\rVert}{#1}
\DeclarePairedDelimiterX{\Avg}[1]{\langle}{\rangle}{#1}
\DeclarePairedDelimiterX{\Inner}[2]{\langle}{\rangle}{#1\,\delimsize\vert\,#2}
\newcommand*{\delimsize}{}

\DeclareMathOperator*{\argmin}{arg\,min}

\DeclareMathOperator{\Diag}{diag}

\DeclareMathOperator{\Prox}{prox}

\newcommand*{\TransposeLetter}{\mathrm{T}}
\newcommand*{\T}{^{\TransposeLetter}}

\newcommand*{\V}[1]{\boldsymbol{#1}}
\newcommand*{\M}[1]{\mathbf{#1}}
\newcommand*{\subM}[1]{\boldsymbol{#1}}

\newcommand*{\Mcal}[1]{{\mathcal{#1}}}

\newcommand{\cX}{{\mathcal X}}
\newcommand{\cC}{{\mathcal C}}

\newcommand{\bR}{{\mathbb R}}
\newcommand{\ps}[1]{\langle #1\rangle}

\newcommand*{\Set}[1]{\mathbb{#1}}

\newcommand*{\Reals}{\Set{R}}

\newcommand*{\Zero}{\V 0}

\usepackage{amsthm}
\theoremstyle{plain} 

\newtheorem{lemma}{Lemma}[section]

\theoremstyle{definition}

\theoremstyle{remark}

\usepackage{url}

\begin{document}

\title{Distributed Deblurring of Large Images of Wide Field-Of-View}

\author{Rahul~Mourya, 
	Andr\'{e}~Ferrari, 
	R\'{e}mi~Flamary, 
	Pascal~Bianchi, 
	and~C\'{e}dric~Richard~ 
	\thanks{The work was supported in part by Agence Nationale pour la Recherche (ODISSEE project, ANR-13-ASTR-0030 and MAGELLAN project, ANR-14-CE23-0004-01).}
	\thanks{Rahul~Mourya and Pascal~Bianchi are with LTCI, T\'{e}l\'{e}com ParisTech, Universit\'{e} Paris-Saclay, F-75013, Paris, France (e-mail:firstname.lastname@telecom-paristech.fr).}
	\thanks{R\'{e}mi~Flamary, Andr\'{e}~Ferrari and C\'{e}dric Richard are with Lagrange Laboratory, Universit\'{e} de Nice Sophia Antipolis, F-06108 Nice, France (email:firstname.lastname@unice.fr, and ferrari@unice.fr).}
	\thanks{A preliminary shorter version of this work has been submitted to the conference EUSIPCO 2017.}}


\markboth{Submitted to IEEE Trans. on Image Processing, 2017}
{Mourya \MakeLowercase{\textit{et al.}}: Distributed Nonblind Deblurring of Large Images of a Wide Field-Of-View}
%

\IEEEpubid{xxxx--xxxx/xx\$00.00~\copyright~2017 IEEE}


\maketitle

\begin{abstract}
Image deblurring is an economic way to reduce certain degradations (blur and noise) in acquired images. Thus, it has become essential tool in high resolution imaging in many applications, e.g., astronomy, microscopy or computational photography. In applications such as astronomy and satellite imaging, the size of acquired images can be extremely large (up to gigapixels) covering wide field-of-view suffering from shift-variant blur. Most of the existing image deblurring techniques are designed and implemented to work efficiently on centralized computing system having multiple processors and a shared memory. Thus, the largest image that can be handle is limited by the size of the physical memory available on the system. In this paper, we propose a distributed nonblind image deblurring algorithm in which several connected processing nodes (with reasonable computational resources) process simultaneously different portions of a large image while maintaining certain coherency among them to finally obtain a single crisp image. Unlike the existing centralized techniques, image deblurring in distributed fashion raises several issues. To tackle these issues, we consider certain approximations that trade-offs between the quality of deblurred image and the computational resources required to achieve it. 
The experimental results show that our algorithm produces the similar quality of images as the existing centralized techniques while allowing distribution, and thus being cost effective for extremely large images. 
\end{abstract}


\section{Introduction}
\label{sec:introduction}
\IEEEPARstart{I}n many applications, it is essential to have high resolution images for precise analysis and inferences from them. However, a certain amount of degradations (blur and noise) are unavoidable in many of the imaging systems due to several factors, such as the limited aperture size, the involved medium (both atmosphere and optics), and the imaging sensors. With advances in the computational technologies, the digital image restoration techniques, such as image deblurring, have been proven to be an economic way to enhance resolution, signal-to-noise ratio, and contrast of the acquired images. In this paper we can consider the situation where we assume that the information about blur and the statistics of noise is know a priori, and then the restoration is referred to as nonblind deblurring. In general, image deblurring is an ill-posed inverse problem \cite{titterington1985general,demoment1989image}. In a Bayesian setting it can be expressed as \emph{maximum-a-posteriori} estimation problem \cite{richardson1972bayesian}, which boils down to the following numerical optimization problem:
\begin{align}
\label{Eq:generic_img_deblurring_problem}
\V{x}^{\ast} := \argmin_{\V{x}} \{ \varPsi_{\text{data}} (\V{y}, \V{x}) + \lambda \ \varPsi_{\text{prior}} (\V{x}) \}
\end{align}
where the vectors $\V{y} \in \Reals^{m}$ and $\V{x} \in \Reals^{n}$ represent the observed (acquired blurry and noisy) image, and the unknown crisp image to be estimated, respectively. Two dimensional (2D) images are represented as column vectors by lexicographically ordering their pixels values. The first term $\varPsi_{\text{data}}$ in (\ref{Eq:generic_img_deblurring_problem}) is called \emph{likelihood} or data-fidelity term that depends upon the noise and the image formation model. The second term $\varPsi_{\text{prior}}$ is called \emph{a priori} or regularizer that imposes any prior knowledge on the unknown image $\V{x}$. The scalar parameter $\lambda$ keeps trade-off between the \emph{likelihood} and \emph{a priori} terms, which generally depends upon the noise level in the observed image. 

\IEEEpubidadjcol

Blur in an acquired image is characterized by the impulse response of the imaging system, commonly known as point-spread-function (PSF). As far
as a narrow field-of-view is concerned, PSF can be considered constant throughout the field-of-view leading to a shift-invariant blur. The
blurring operation in this case is modeled as simple a convolution between the sharp image and the PSF, and can be performed efficiently in Fourier
domain. In many cases, however, blur varies throughout the field-of-view due to several causes: relative motion between the camera and the scene; moving objects with respect to the background; variable defocussing of non-planar scenes with some objects located in front or behind the in-focus plane; optical aberrations such as space-variant distortions, vignetting or phase aberrations. In here we consider shift-variant blur across the field-of-view of a planar scene, i.e. all the objects in the scene lie in-focus plane; see \cite{porter1984compositing} for discussion on shift-variant blur in the case of non-planar scene. Two subtly different situations can be considered for shift-variant blur in planar scene: i) a large field-of-view is captured into several pieces by multiple imaging systems each capturing a small portion of the whole scene, and ii) the same scene is captured by a single imaging system having wide field-of-view. 
In the former case, each imaging system can have slightly different blur from each other operating under different settings. Thus, the captured images can have slightly different blurs from each other resulting into piece-wise constant blur in the whole scene. We will refer to this situation as piece-wise constant shift-variant blur. In the latter case, the blur in the whole field-of-view can vary smoothly, e.g. blur due to optical aberrations. We will refer to this situation as smooth shift-variant blur. The shift-invariant and piece-wise constant shift-variant blurs can be considered as special cases of the smooth shift-variant blur. The blurring operation in the smooth shift-variant case cannot be modeled by a simple convolution, and there does not exists any efficient and straight-forward way to perform the operation. However, there are some fast approximations of smooth shift-variant blur operators proposed in \cite{Nagy1998,Hirsch2010,Denis2011}; see \cite{Denis2015} for detailed comparison between different approximations. 

If we approximate the noise in observed image by a non-stationary white Gaussian noise as considered in \cite{mugnier2004mistral} (see \cite{Lanteri2005, Benvenuto2008, Chouzenoux2015} for more refined noise model), then a discrete image formation model can
be written as:
\begin{equation}
\label{Eq:image_formation}
\V{y} = \M{H} \ \V{x} + \V{\varepsilon}
\end{equation}
where the matrix $\M{H} \in \Reals^{m\times n}, \ m < n$, denotes the blur operator\footnote{Blur operator is a rectangular matrix since size of observed image is restricted by the physical size of the image sensor, i.e., sensor captures little less than what optics can see.}, and $\V{\varepsilon}$ is the zero-mean non-stationary white Gaussian noise, i.e. $\V{\varepsilon}(\ell) \sim \mathcal{N}(0, \V{\sigma}^{2}(\ell))$ with $\V{\sigma}^{2}(\ell)$ denoting the noise variance at $\ell$th component of the vector $\V{y}$. We assume that the raw captured image is preprocessed to yield an image that closely follows the above image formation model (\ref{Eq:image_formation}). The preprocessing may includes the correction of the background and of the flat field, the correction of defective pixels and possibly of its correlated noise, and the scaling of the image in photons.

Before we proceed further, let us introduce some more notations. Hereafter, we will use upper-case bold and lower-case bold letters for denoting matrices (and linear operators), and column vectors, respectively. For every set $\Mcal{T}$, we denote by $\arrowvert \Mcal{T}\arrowvert$ the cardinality of the set, and by $\Reals^{\Mcal{T}}$ the set of functions on $\Mcal{T}\to\Reals$. Let $\Mcal{X}$ represents Euclidean space, then we denote by $\ps{\,.\,,\,.\,}$ the standard inner product on $\Mcal{X}$, and by $\|\,.\,\|$ the Euclidean norm. Letting $\M{V}$ denote a positive definite linear operator of $\cX$ onto itself, we use the notation $\ps{\V{x},\V{y}}_\M{V}=\ps{\V{x},\M{V}\V{y}}$, and we denote by $\|\,.\,\|_\M{V}$ the corresponding norm. When $\M{V}$ is a diagonal operator of the form $(\M{V}\V{x}):\ell \mapsto {\V{\alpha}}(\ell) \ {\V{x}}(\ell)$ we equivalently denote  $\|\V{x}\|_\M{V}^2$ by $\|\V{x}\|_{\V{\alpha}}^{2}=\sum_\ell \V{\alpha}(\ell)\V{x}(\ell)^2$, where $\V{\alpha} = (\alpha_1, \alpha_2,\cdots)$. We also denote $\| \V{x}\|_{1,\M{V}}=\sum_\ell \V{\alpha}(\ell) \ |\V{x}(\ell)|$ the weighted $L_1$-norm, simply noted $\|\V{x}\|_1$ when $\M{V}$ is the identity. Moreover, if $\M{L}$ is a linear operator, we denote by $\M{L}^{\T}$ its adjoint operator. 

Considering the above image formation model (\ref{Eq:image_formation}) with a known PSF, the (nonblind) image deblurring problem (\ref{Eq:generic_img_deblurring_problem}) can be explicitly expressed as:
\begin{align}
\label{Eq:specific_img_deblurring_problem}
\V{x}^{\ast} := \argmin_{\V{x} \geq \Zero} \{ \frac{1}{2} \| \V{y} - \M{H} \ \V{x} \|^{2}_{\M{W}} + \lambda \ \phi(\M{D} \ \V{x}) \}
\end{align} 
where the matrix $\M{W}$ is diagonal with its components given by $\M{W}(\ell,\ell) = 1 / \V{\sigma}^{2}(\ell)$ for observed pixels and $\M{W}(\ell,\ell) = 0$ for unmeasured pixels. The function $\phi$ represents some regularizer (e.g., Tikhonov's $L_{2}$-norm, sparsity promoting $L_{1}$-norm, or Huber's mixed $L_{1}$-$L_{2}$-norm), and $\M{D}$ is some linear operator, e.g., finite forward difference or some discrete wavelet transform. The formalism (\ref{Eq:specific_img_deblurring_problem}) of image deblurring problem is referred to as analysis-based approach where the unknown variable is expressed in the image domain itself. An alternative formalism, referred to as synthesis-based approach, is also considered in literature where the unknown variable is in some transformed domain, e.g., coefficients of some discrete wavelet; see \cite{Elad2007} for detailed discussion on the two approaches and their comparisons. Without loss of generality, in this paper we consider only the analysis-based approach for expressing the image deblurring problem.

Depending upon the structures of the two terms in the optimization problem (\ref{Eq:specific_img_deblurring_problem}), the solution can be obtained in a single step i.e., a closed-form solution (e.g., Wiener deconvolution), or one has to rely on iterative solvers, which may require, at each iteration, the same computational expense or even more than the closed-form solution\footnote{Applying blur operator $\M{H}$ or its adjoint $\M{H}^{\T}$ is an expensive operation, and one may require to apply them several times per iterations of the solver.}. The latter one is the frequently occurring case in many applications, e.g., the regularization term is not differentiable (nonsmooth). Though, there exists a vast literature on numerical optimization, but we refer the readers to \cite{Beck2009,Afonso2010,Matakos2013,Mourya2015} and the references therein for some of the recent and fast optimization algorithms specially developed for inverse problems in imaging.

\subsection{Motivation: Deblurring Extremely Large Image}
\label{subsec:problem}
With the advances in imaging technologies, the applications in astronomy, satellite imagery and others are able to capture extremely large image of wide field-of-view. The size of such images can vary from few tens megapixels up to gigapixels. As discussed previously, we consider two imaging situations where such large images are acquired: piece-wise constant and smooth shift-variant blur. In the former situation, the images from multiple narrow field-of-view imaging systems are supposed to be mosaicked together into a single image provided that their relative positions within the whole field-of-view are known. For further simplification, we assume that all the narrow field-of-view imaging systems have the same noise performance. In the both imaging situations, it is highly desirable to obtain a single crisp image from the acquired image(s) for any further applications. 

Image deblurring is well studied topic and a vast literature exists \cite{richardson1972bayesian, titterington1985general, demoment1989image, Thiébaut2005, hansen2006deblurring, Benvenuto2008, Figueiredo2009, Matakos2013} for moderate size (few megapixels) images with shift-invariant blur.
These methods differ from each other in the way they formulate the data-fidelity term or the regularization term, and the optimization algorithms they propose to solve the resulting problem. The image deblurring problem becomes more complicated when one considers the images suffering from shift-variant blur. Some recent works \cite{Nagy1998,Hirsch2010,Denis2011,Denis2015} proposed efficient methods for deblurring images with shift-variant blur. However, all the methods listed above are designed and implemented to work efficiently on a centralized computing system, possibly with multiple processor cores sharing a single memory space. Such a computing system is commonly referred to as multi-threaded shared memory system based on single instruction multiple data (SIMD) architecture. Hereafter, we will refer to all such existing methods by \emph{centralized} deblurring techniques. In contrast to deblurring a moderate size image, deblurring an extremely large image is a challenging problem since the largest size of the image that can be handled is then limited by the capacity of shared memory available on the centralized system. It is not cost effective to build a centralized system with several processor cores and a huge shared memory when the modern distributed computing systems (consist of several connected processing nodes each having reasonable amount of computational resources) are proving to be far more economical way for solving huge-scale problems than the centralized system. Many domains such as machine learning, data mining \cite{bekkerman2011scaling}, and others are already benefiting from distributed computing approach for discovering patterns in huge datasets distributed at different nodes. As per our knowledge, there are very few works, e.g., \cite{cui2013distributed, wang2013distributed, ferrari2014distributed, meillier2016two}, considering distributed computing approach for image restoration/reconstruction problems. Within this context, we propose a distributed image deblurring algorithm for large images acquired from either of the imaging situations mentioned above.

\subsection{Our contributions}
\label{subsec:contributions}
In this paper, we propose a distributed algorithm for deblurring large images, which needs less frequent inter-node communication. Our algorithm can handle the acquired image(s) from both the imaging situations: smooth shift-variant and piece-wise constant shift-variant blur. For the first imaging situation, we consider splitting the large image into sufficiently small overlapping blocks and then deblur them simultaneously on the several processing nodes while maintaining certain coherencies among them so as to obtain a single crisp image without any transition artifacts among the deblurred blocks. For the second imaging situation, we assume that the relative positions of narrow images in the whole field-of-view are known, and there are certain overlaps among them. If a narrow image is not sufficiently small to fit in the memory of a node, then we can split it further into smaller overlapping blocks. Similar to the first situation, we deblur the small images simultaneously to obtain a single crisp image. To do so, we reformulate the image deblurring problem (\ref{Eq:generic_img_deblurring_problem}) into a distributed optimization problem with consensus, and then present an efficient optimization method to solve it. Our distributed deblurring algorithm is rather generic in the sense that it can handle different situations such as shift-invariant or shift-variant blur with any possible combination of the data-fidelity and regularization terms. Depending on the structures of the data-fidelity and regularizer, we can select any fast optimization algorithm for solving the local deblurring problems at processing nodes. By several numerical experiments, we show that our algorithm is cost effective in term of computational resources for large images, and is able obtain the similar quality of deblurred images as the existing \emph{centralized} deblurring techniques that are applicable only to images of moderate size.

The remaining parts of the paper is organized as follows. In Section \ref{sec:distributed_image_deblurring}, we discuss the difficulties associated with distributed formulations of image deblurring problem, and the possible approaches to overcome it. In Section \ref{sec:proposed_approach}, we present the distributed formulation and an efficient distributed algorithm to solve it. We discuss the criteria for selecting the different parameters associated with our approach. In Section \ref{sec:numerical_experiments_results}, we discuss the
implementation details and present several numerical experiments where we compare the results obtained from different approaches. Finally, in
Section \ref{sec:conclusion}, we conclude our work with possible future enhancements.

\section{Distributed Computing Approach for Image Deblurring}
\label{sec:distributed_image_deblurring}
A generic approach for dealing with large-scale problem is the ``divide and conquer'' strategy, in which a large-scale problem is decomposed into smaller subproblems in such way that solving them and assembling theirs results would produce the expected final result or, at least, reasonably close to it. Distributed computing approach has emerged as a framework of such a strategy. Distributed computing systems are consist of several processing nodes, each having a reasonable amount of computational resources (in term of memory and processor cores), connected together via some communication network so that nodes can exchange messages to achieve a certain common goal. They are built based on the multiple instructions multiple data (MIMD) architecture. Nodes in a distributed system may not be necessarily located at the same physical location, so a high speed communication links may not be always feasible among them. Thus, an efficient distributed algorithm is the one which is computation intensive rather than communication demanding. In many applications such as machine learning, data mining, distributed approaches have become de-facto standard for efficiently estimating extremely large number of parameters using huge dataset distributed on the different nodes. Taking such an inspiration, one can devise a distributed strategy for deblurring extremely large images. A possible approach is to use the distributed array abstraction available on modern distributed systems, and reimplement the standard \emph{centralized} deblurring techniques using distributed arrays instead of a shared memory array. However, the bottleneck of such an approach would be extensive data communication among the nodes at each iteration of the optimization algorithm. To overcome this, a straightforward approach would be to split the given observed image $\V{y}$ into $N$ smaller contiguous blocks $\V{y}_{i} \in \Reals^{m_{i}}, i=1,2,\cdots,N $, and deblur them independently to obtain the deblurred blocks $\hat{\V{x}}_{i} \in \Reals^{n_{i}}, i=1,\cdots,N$ that can be merged to obtain the single deblurred image $\hat{\V{x}} \in \Reals^{n}$. However, there may arise several issues from both the theoretical and practical point-of-views as discussed below. 

\subsection{Issues with Distributed Approach for Image Deblurring}
\label{subsec:the_difficulties}
For sake of simplicity, let us first consider the shift-invariant image deblurring problem, and point out some of the major issues:
\begin{enumerate}
	\item Practically, it is infeasible to explicitly create the blur operator matrix $\M{H}$. For this reason blur operation is efficiently performed in Fourier domain using Fast Fourier Transforms (FFT) algorithm. Thus, there is no straight forward way to split the matrix $\M{H}$ into block matrices $\subM{H}_{i} \in \Reals^{m_{i}\times n_{i}}$ so that they could be used to perform independent deblurring of the observed blocks $\V{y}_{i}$ such that the final result would be equivalent to using the original matrix $\M{H}$ on whole image $\V{y}$. Any approximation of block matrices $\subM{H}_{i}$ chosen not appropriately may create artifacts such as nonsmooth transition among the deblurred blocks.
	\item Blur operation performed in Fourier domain assumes circular boundary condition. The circulant blur assumption in image deblurring process can lead to severe artifacts due to discontinuity at boundaries caused by the periodic extension of image \cite{Matakos2013}. Thus, deblurring the individual blocks $\V{y}_{i}$ can introduce boundary artifacts into them that will eventually worsen the quality of final deblurred image.
	\item Splitting the image into blocks can also raise issues with the regularization term. Many frequently used regularizers are not separable, i.e., its value on the whole image is not equivalent to the sum its value on the blocks. Such an approximation can result into structural incoherencies among the deblurred blocks.
\end{enumerate}
The issues mentioned above complicate further the shift-variant image deblurring problem, since there is no any efficient and straightforward way to formulate shift-variant blur operator. Considering the above issues associated with the na\"{\i}ve ``split, deblur independently, and merge'' approach, it is obvious that one needs to approximate the problem (\ref{Eq:specific_img_deblurring_problem}) in a way so that it can be solved efficiently in a distributed manner, yet obtain the solution that remains reasonably close to the solution of the original problem (\ref{Eq:specific_img_deblurring_problem}). In the next section, we describe how to tackle the above issues, and then we will present our distributed image deblurring algorithm.

\subsection{Tackling the Issues in Distributed Image Deblurring}
\label{subsec:tackling_the_difficulties}
As discussed above, it is nontrivial problem to explicitly create a blur operator and then decompose it into block matrices to operate them independently for distributed image deblurring. However, we can formulate an approximation of the block matrices $\subM{H}_{i}$ such that when they are used in deblurring process, certain homogeneities among the deblurred blocks are imposed. Let us consider the smooth shift-variant blur case. Provided that we are able to sample $N$ local PSFs $\V{h}_{i}, i=1,\cdots, N$ at regular grid points within the field-of-view, the approximation of shift-variant blur operator proposed in \cite{Hirsch2010,Denis2011} suggests an interesting idea to formulate the block matrices $\subM{H}_{i}$. Their approximations of the shift-variant blur operator is based upon the idea that a PSF at any point within field-of-view can be well approximated by linear combination (interpolation) of the PSFs sampled at the neighboring grid points. Said so, the shift-variant blurring operation can be written as:
$\V{y} = \M{R} \sum_{i=1}^{N} \subM{Z}_{i} \subM{H}_{i} \subM{W}_{i} \subM{C}_{i} \V{x}$, where $\subM{C}_{i}$ are chopping operators that select overlapping blocks from the image $\V{x}$, $\subM{W}_{i} = \Diag(\V{\omega}_{i})$ are interpolation weights corresponding to each blocks, $\subM{H}_{i}$ are blur operators corresponding to the sampled local PSFs $\V{h}_{i}$, $\subM{Z}_{i}$ are operators that zero-pads the corresponding blurred blocks keeping it in the same relative position with respect to the whole image, and $\M{R}$ is a chopping operator that restrict the final blurred image to sensor size. Interpolation weights $\V{\omega}_{i}$ are non-zero only within certain locality depending upon the order of interpolation. This approximation allows to perform shift-variant blurring combining several local shift-invariant blurring, which can be performed efficiently in Fourier domain. This approximation renders smooth variation of blur throughout the whole image. Higher accuracy in blurring operation can be achieved by denser sampling of PSFs with the field-of-view. The image deblurring results presented in \cite{Hirsch2010, Denis2015} suggest that the first-order interpolation (e.g., $\V{\omega}_{i}$ is a ramp within the range $[0, 1]$ between the two adjacent grid points for 1D signal as illustrated in Fig.\ref{fig:demonstration1}) is sufficient to achieve reasonably good quality of deblurred images. Using first-order interpolation, the shift-variant blur operator is only four times expensive than the shift-invariant blur operation on the same size of image. However, this fast approximation of shift-variant blur operator is efficient only for a centralized multi-threaded shared memory implementation. It is not directly applicable for distributed setting as it will lead to intensive data communication among the nodes each time the blur operator or its adjoint is used by an iterative solver. Nevertheless, this approximation suggests the following idea for distributed image deblurring: i) split the observed image $\V{y}$ into $N$ overlapping blocks $\V{y}_{i}$, ii) generate 2D first-order interpolation weights $\V{\omega}_{i}$ corresponding to each block, iii) provided with $N$ PSFs $\V{h}_{i}$ sampled locally within each block, distribute the observed blocks, the PSFs, and the interpolation weights among the $N$ nodes, and iv) then on each node perform local image deblurring while maintaining certain consensus among the overlapping pixels of the adjacent blocks. The consensus operation can be a weighted averaging among the overlapping pixels. Using the interpolation weights $\V{\omega}_{i}$ for averaging operation will indirectly impose the smooth variation of blur across whole observed image. Without loss of generality, the aforementioned strategy is also applicable to the cases when an image suffers from shift-invariant or piece-wise constant shift-variant blur since these two operations can be well approximated by the above fast shift-variant blur operator for the smooth blur variation.

As pointed out above, performing local deblurring in Fourier domain may lead to boundary artifacts, thus to avoid such artifacts in the deblurred blocks, we borrow the idea from \cite{Matakos2013}. We express the local blur operator as: $\subM{H}_{i} = \boldsymbol{\mathcal{C}}_{i} \boldsymbol{\mathcal{H}}_{i}, i=1,\cdots, N$, where $\boldsymbol{\Mcal{H}}_{i} \in \Reals^{n_{i} \times n_{i}}$ are circular convolution matrices formed from PSFs $\V{h}_{i}$, and $\boldsymbol{\mathcal{C}}_{i} \in \Reals^{m_{i} \times n_{i}}$ are chopping operators, which restrict the convolution results to the valid regions. Expressing the local blur operators in this form allows an efficient non-circulant blur operation in Fourier domain while suppressing boundary artifacts in the deblurred patches.

Concerning the last issue about approximating regularizer on the whole image by sum of it on blocks, it depends upon the structure of regularizer. If the regularizer is separable in image domain, e.g., $\phi(\M{D} \V{x}) = \| \V{x} \|_{1} \ \text{or} \ \| \V{x} \|^{2}_{2}$, then the aforesaid approximation holds. However, for many of the frequently used regularizers, e.g., $\|\M{D}\V{x} \|^{2}_{2}$ or $\| \M{D}\V{x} \|_{1}$ or $\| \M{D}\V{x} \|_{2}$ where $\M{D}$ represents finite-forward difference operator or some discrete wavelet transform, the aforementioned approximation does not hold. This can lead to incoherencies among the deblurred blocks, which can render nonsmooth transitions among them. However, as shown hereafter a sufficient amount of overlaps among the adjacent blocks with aforementioned consensus imposed on the overlapping pixels during deblurring process will limit the incoherencies, and suppress any nonsmooth transition in the final deblurred image.

\section{The Proposed Approach}
\label{sec:proposed_approach}
Considering all the ideas developed in Section \ref{sec:distributed_image_deblurring}, we propose a generic framework for distributed image deblurring applicable to images suffering from both the shift-invariant and shift-variant blur. We make some reasonable approximations, and reformulate the original problem (\ref{Eq:specific_img_deblurring_problem}) into the distributed optimization problem presented below. The resulting optimization problem is solved in a distributed manner by Douglas-Rachford (D-R) splitting algorithm that first appeared in \cite{Lions1979}.
\subsection{General Setting}
Consider a distributed computing system with a set of $N$ nodes having peer-to-peer bidirectional connections among them, or at least connection topology shown in Fig. \ref{fig:demonstration2}. Given an observed image $\V{y} \in \Reals^{m}$, we split it into $N$ blocks $\V{y}_{i} \in \Reals^{m_{i}}, \forall i=1,2,\cdots,N$, with certain overlaps among them. We generate 2D first-order interpolation weights $\V{\omega}_{i}$ corresponding to the observed blocks as shown in Fig.~\ref{fig:shiftinvariant_simulation}(e). Provided that we are able to sample $N$ PFSs $\V{h}_{i}$ within the blocks, and $\subM{H}_{i} \in \Reals^{m_{i} \times n_{i}}$ be the corresponding blur operators, we distribute the observed blocks, the corresponding PSFs, and the interpolation weights among the nodes. We, then, seek to distributively estimate the whole unknown crisp image $\V{x} \in \Reals^{n}$. Let us denote by $\Mcal{P}_1,\dots,\Mcal{P}_N$ a collection of $N$ subsets of $\{1,\dots,n\}$. For every $i=1,\dots, N$, we assume that $i$th compute node is in charge of estimating the components of $\V{x}$ corresponding to the indices $\Mcal{P}_{i}$. The subsets $\Mcal{P}_1,\dots, \Mcal{P}_N$ are overlapping. Hence, different nodes handling a common component of $\V{x}$ must eventually agree on the value of the latter. Formally, we introduce the product space $\Mcal{X}:=\Reals^{\Mcal{P}_1}\times\cdots\times\Reals^{\Mcal{P}_N}$, and we denote by $\Mcal{C}$ the set of vectors $(\V{x}_1,\dots,\V{x}_N)\in \Mcal{X}$ satisfying the restricted consensus condition
$$
\forall (i,j)\in\{1,\dots,N\}^2, \, \forall \ell\in \Mcal{P}_i\cap \Mcal{P}_j,\, \V{x}_i(\ell)= \V{x}_j(\ell)\,.
$$
Moreover, we assume that every $i$th node is provided with a local convex, proper and lower semicontinuous  function $f_i:\Reals^{\Mcal{P}_i}\to (-\infty,+\infty]$. We consider the following constrained minimization problem on $\Reals^{\Mcal{P}_1}\times\cdots\times\Reals^{\Mcal{P}_N}$:
\begin{equation}
\label{Eq:genericpb}
\argmin_{\V{x}_1\cdots \V{x}_N} \sum_{i=1}^Nf_i(\V{x}_i)\ \text{s.t. } {(\V{x}_1,\dots,\V{x}_N)\in \Mcal{C}}\,.
\end{equation}
For our image deblurring problem, the local cost function $f_{i}$ is composed of the local data-fidelity term $$\V{x}_{i}\mapsto \frac 12 \|\V{y}_i - \subM{H}_{i} \ \V{x}_{i} \|^2_{\M{W}_i},$$ for some positive definite $\M{W}_i(\ell, \ell) = 1 / {\V{\sigma}_{i}^{2}(\ell)}$ as in (\ref{Eq:specific_img_deblurring_problem}), and a regularizer $\phi_{i}( \M{D}_{i} \ \V{x}_{i})$ with positivity constraint on $\V{x}_{i}$, i.e., $\V{x}_{i} \in \Reals^{\Mcal{P}_i}_{+}$. If $\Mcal{A}$ is a set, the notation $\iota_{\Mcal{A}}$ stands for the indicator function of the set $\Mcal{A}$, equal to zero on that set and to $+\infty$ elsewhere. Thus, the local cost function needed to be minimized at each of the nodes have the form:
$$
f_i(\V{x}_{i}) = \frac 12\|\V{y}_i - \subM{H}_{i} \ \V{x}_{i} \|^2_{\M{W}_i} + \lambda_{i} \ \phi_{i}(\M{D}_{i} \ \V{x}_{i}) + \iota_{\Reals^{\Mcal{P}_i}_{+}}\left(\V{x}_{i}\right)
$$
where $\phi_{i}$ are convex, proper and lower semicontinuous functions and $\M{D}_{i}$ are linear operators on $\Reals^{\Mcal{P}_i}$.
\subsection{Optimization Algorithm}
\label{subsec:optimization_algorithm}
Before we present our distributed optimization algorithm for solving problem (\ref{Eq:genericpb}), let us introduce one more notation. For any convex, proper and lower semicontinuous function $h:\Mcal{X}\to(-\infty,+\infty]$, we introduce the \emph{proximity operator}
$$
\Prox_{\M{V}^{-1},h}(\V{v}) = \argmin_{\V{w}\in \Mcal{X}} h(\V{w})+\frac{\|\V{w}- \V{v}\|_\M{V}^2}{2}
$$
for every $\V{v}\in \Mcal{X}$.

For solving (\ref{Eq:genericpb}), we consider D-R Splitting algorithm, thus we reformulate (\ref{Eq:genericpb})
as 
$$
\argmin_{\V{x}\in \Mcal{X}} f(\V{x})+g(\V{x})
$$
where $g=\iota_{\Mcal{C}}$ is the indicator function of $\Mcal{C}$ and $f(\V{x})=\sum_if_i(\V{x}_i)$ for every $\V{x}=(\V{x}_1,\dots,\V{x}_N)$ in $\cX$. 
Let us equip the Euclidean space $\Mcal{X}$ with the inner product $\ps{\,.\,,\,.\,}_\M{V}$ for some positive definite linear operator $\M{V}:\Mcal{X}\to\Mcal{X}$. 
Let $\rho^{(k)}$ be a sequence in $]0, 2[$, and $\V{\epsilon}^{(k)}_{1}$ and $\V{\epsilon}^{(k)}_{2}$ be sequences in $\Mcal{X}$, then the D-R splitting algorithm writes as:
\begin{align*}
\V{x}^{(k+1)} &=\Prox_{\M{V}^{-1},f}(\V{u}^{(k)}) + \V{\epsilon}^{(k)}_{1} \\
\V{z}^{(k+1)} &=\Prox_{\M{V}^{-1},g}(2\V{x}^{(k+1)}- \V{u}^{(k)}) + \V{\epsilon}^{(k)}_{2} \\
\V{u}^{(k+1)} &= \V{u}^{(k)} + \rho^{(k)} \left( \V{z}^{(k+1)}-\V{x}^{(k+1)} \right) \;.
\end{align*}
If the following holds:
\begin{enumerate}
	\item the set of minimizers of (\ref{Eq:genericpb}) is non-empty,
	\item $ \Zero \in \mathrm{ri}(\mathrm{dom}(f)-\cC)$, where $\mathrm{ri}$ represents relative interior,
	\item $\sum_{k} \rho^{(k)}(2 - \rho^{k}) = +\infty$,
	\item and $\sum_{k} \| \V{\epsilon}^{(k)}_{1} \|_{2} + \| \V{\epsilon}^{(k)}_{2} \|_{2} < +\infty$,
\end{enumerate}
then the iterates $\V{u}^{(k)}$ converges weakly to some point in $\Mcal{X}$ and $\V{x}^{(k)}$ converge to a minimizer of (\ref{Eq:genericpb}) as $k\to\infty$; see \cite[Corollary 5.2]{combettes2004solving} for the proof. The parameter $\rho^{(k)}$ is referred to as a relaxation factor, which can be tuned to improve the convergence. The sequences $\V{\epsilon}^{(k)}_{1}$ and $\V{\epsilon}^{(k)}_{2}$ allow some perturbations in the two $\Prox$ operations, which is very useful in the cases when the $\Prox$ operations do not have closed-form solutions and have to rely on some iterative solvers. 

From now onward, we assume that $\M{V}$ is a diagonal operator of the form $\M{V}\V{x} =(\M{V}_1\V{x}_1,\dots,\M{V}_N \V{x}_N)$, where for every $i$, 
\begin{eqnarray*}
	\M{V}_{i} \V{x}_{i}\,:\, \Mcal{P}_{i} &\to&\Reals \\
	\ell&\mapsto& \V{\alpha}_{i}(\ell)\,\V{x}_{i}(\ell)\,,
\end{eqnarray*}
where $\V{\alpha}_i(\ell)$ is a positive coefficient to be specified later.
For every $\ell\in\{1,\dots,n\}$, we introduce the set $\Mcal{P}_\ell^-=\{i\,:\,\ell\in \Mcal{P}_i\}$.
\begin{lemma}
	For every $\V{x}\in\Mcal{X}$, the quantity $\V{z}=\Prox_{\M{V}^{-1}, \iota_{\cC} }(\V{x})$ is such that for every 
	$i\in\{1,\dots,N\}$ and every $\ell\in \Mcal{P}_i$,
	$$
	\V{z}_i(\ell) = \frac{\sum_{i\in \Mcal{P}^-_\ell}\V{\alpha}_i(\ell) \ \V{x}_i(\ell)}{\sum_{i\in \Mcal{P}^-_\ell}\V{\alpha}_i(\ell)}\,. 
	$$
\end{lemma}
\begin{proof}
	For every $\V{x}\in\Mcal{X}$ of the form $\V{x}=(\V{x}_1,\dots,\V{x}_N)$, and every $\ell\in \{1,\dots,n\}$, we use the notation
	$\V{x}_{\Mcal{P}^-_\ell}(\ell) = (\V{x}_i(\ell)\,:\,i\in \Mcal{P}^-_\ell)$. We denote by $\cC_\ell$ the linear span of the vector 
	$(1,\dots,1)$ in $\bR^{|\Mcal{P}^-_\ell|}$, where $|\Mcal{P}^-_\ell|$ is the cardinality of $\Mcal{P}^-_\ell$. 
	The function $g=\iota_\cC$ writes
	$$
	g(\V{x}) = \sum_{\ell=1}^n \iota_{\cC_\ell}(\V{x}_{\Mcal{P}^-_\ell}(\ell))\,.
	$$
	For an arbitrary $\V{x}\in \Mcal{X}$, the quantity $\V{z}=\Prox_{\M{V}^-,g}(\V{x})$ is given by
	\begin{align*}
	\V{z} &= \arg\min_{\V{w}\in \Mcal{X}} \sum_{\ell=1}^{n} \Bigl( \iota_{\cC_\ell}(\V{w}_{\Mcal{P}^-_\ell}(\ell)) \Bigr. \\
	& \qquad \qquad \qquad \qquad \Bigl. + \frac{1}{2}\sum_{i\in \Mcal{P}^-_\ell}\V{\alpha}_i(\ell)(\V{w}_i(\ell)-\V{x}_i(\ell))^2 \Bigr)
	\end{align*}
	Clearly, for every $\ell\in \{1,\dots,n\}$, the components $\V{z}_i(\ell)$ are equal for all $i\in \Mcal{P}^-_\ell$, to some constant w.r.t. $i$, say $\bar{\V{z}}(\ell)$.
	The later quantity is given by
	\begin{align*}
	\bar{\V{z}} (\ell) &= \arg\min_{\V{w}\in \bR} \sum_{i\in P^-_\ell}\V{\alpha}_i(\ell)(\V{w}-\V{x}_i(\ell))^2 \\
	&= \frac{\sum_{i\in \Mcal{P}^-_\ell}\V{\alpha}_i(\ell)\V{x}_i(\ell)}{\sum_{i\in \Mcal{P}^-_\ell}\V{\alpha}_i(\ell)}\,.
	\end{align*}
\end{proof}

By the above lemma, the D-R splitting algorithm can be explicitly written as follows. For every $ i\in \{1,\dots,N\}$ and every $\ell\in
\{1,\dots,n\}$,
\begin{align*}
\V{x}^{(k+1)}_i &= \Prox_{\M{V}_i^{-1},f_i}(\V{u}_i^{(k)}) \\
\bar{\V{z}}^{(k+1)}(\ell) &= \frac{\sum_{i\in \Mcal{P}^-_\ell}\V{\alpha}_{i}(\ell)(2\V{x}_i^{(k+1)}(\ell)-\V{u}_i^{(k)}(\ell))}{\sum_{i\in \Mcal{P}^-_\ell}\V{\alpha}_i(\ell)} \\
\V{u}^{(k+1)}_i(\ell) &= \V{u}^{(k)}_i(\ell) + \rho^{(k)} \left( \bar{\V{z}}^{(k+1)}(\ell)-\V{x}^{(k+1)}_i(\ell) \right).
\end{align*}
The resulting algorithm is a synchronous distributed algorithm without any explicit master node. It is depicted in Algorithm~\ref{algo:proposed_algo}. Hereafter, we will refer to it by \emph{proposed} deblurring method. Given some initial guess of the local solutions at each node, the first step of Algorithm~\ref{algo:proposed_algo} execute the \textsc{Local-Solver} in parallel on all the nodes to obtain local deblurred blocks. Then all the nodes synchronize at the start of the second step, and then they exchange the overlapping pixels with its adjacent nodes to distributively perform the weighted averaging of those pixels. Then, the last step of the algorithm is executed in parallel on all the nodes.

The convergence speed of \emph{proposed} deblurring method is dependent upon the parameters $\V{\alpha}_{i}$, and like in other optimization algorithms e.g., ADMM \cite{Boyd2011}, selecting the optimal values of $\V{\alpha}_{i}$ for fast convergence is a tedious task. We select $\V{\alpha}_{i} = \gamma \; \V{\omega}_{i}$ for $\gamma > 0$, so that we can tune $\gamma$ for fast convergence, and as well impose the smooth variation of the blur among adjacent blocks.

\begin{algorithm}[t]
	\begin{algorithmic}
		\Procedure{Distributed-Solver}{} 
		\State{Initialize: $\V{u}_i\gets \V{u}^{(0)}_{i}, \forall i =1,2,\cdots,N$} 
		\While{not converged} 
		\For {$i =1\dots N$}
		\State $\V{x}_i\gets$ \textsc{Local-Solver}$(\V{u}_i\,;\V{\alpha}_i,f_i)$
		\EndFor
		\For {$\ell=1\dots n$}
		\State Compute distributively at nodes $i\in \Mcal{P}^-_\ell$:
		\State $\bar{\V{z}}_{i}(\ell)\gets \sum_{i\in P^-_{\ell}}\V{\omega}_i(\ell) (2\V{x}_i(\ell)-\V{u}_i(\ell)), \forall \ell \in  \Mcal{P}_i$
		\EndFor
		\For {$i =1\dots N$}
		\State $\V{u}_i(\ell)\gets  \V{u}_i(\ell) + \rho (\bar{\V{z}}_{i}(\ell)-\V{x}_i(\ell)), \forall \ell \in \Mcal{P}_i$
		\EndFor
		\EndWhile\label{euclidendwhile}
		\State \textbf{return} $\V{x}_1,\dots,\V{x}_N$
		\EndProcedure
		\Procedure{Local-Solver}{$\V{u}\,;\V{\alpha},f$} 
		\State $\V{w}\gets \Prox_{\V{\alpha}^{-1},f}(\V{u}) = \arg\min_{\V{w}} \{ f(\V{w}) + \frac 12 \|\V{w}-\V{u}\|_{\V{\alpha}}^2 \}$
		\State \textbf{return} $\V{w}$
		\EndProcedure
		\caption{Distributed Image Deblurring}
		\label{algo:proposed_algo}
	\end{algorithmic}
\end{algorithm}

\begin{figure*}
	\centering
	{\includegraphics[width=0.95\linewidth]{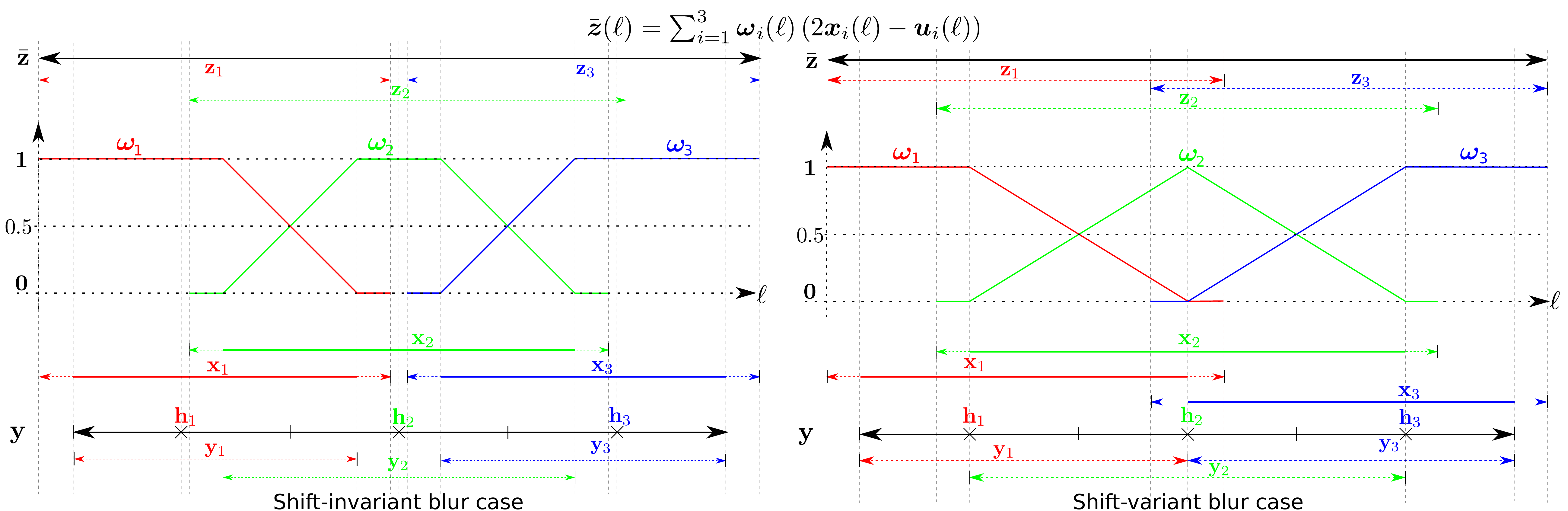}}
	\caption{An 1D illustration showing the splitting of the observed image, the extent of the overlap, and shape of the interpolation weights for the shift-invariant and smooth shift-variant blur cases. Given an observed image $\V{y}$, the crossbar are equidistant reference grid points where PSFs $\V{h}_{i}$ are sampled. Depending upon the case: shift-invariant or shift-variant blur, the extent of the overlap is selected. The observed image is split into $3$ overlapping patches $\V{y}_{i}$. $\V{x}_{i}$ are the locally deblurred patches at $i$th node from the corresponding observed patches $\V{y}_{i}$. The dashed parts at the ends of each $\V{x}_{i}$ are the extra pixels estimated at boundaries assuming no measurements were available for them. $\V{\omega}_{i}$ are the interpolation weights corresponding to the support of $\V{x}_{i}$, and its values are within range $[0, 1]$ such that $\sum_{i=1}^{3} \V{\omega}_{i}(\ell) = 1, \forall \ell=1,\cdots,n$.
	}
	\label{fig:demonstration1}
\end{figure*}

\begin{figure}
	\centering
	\includegraphics[width=0.75\linewidth]{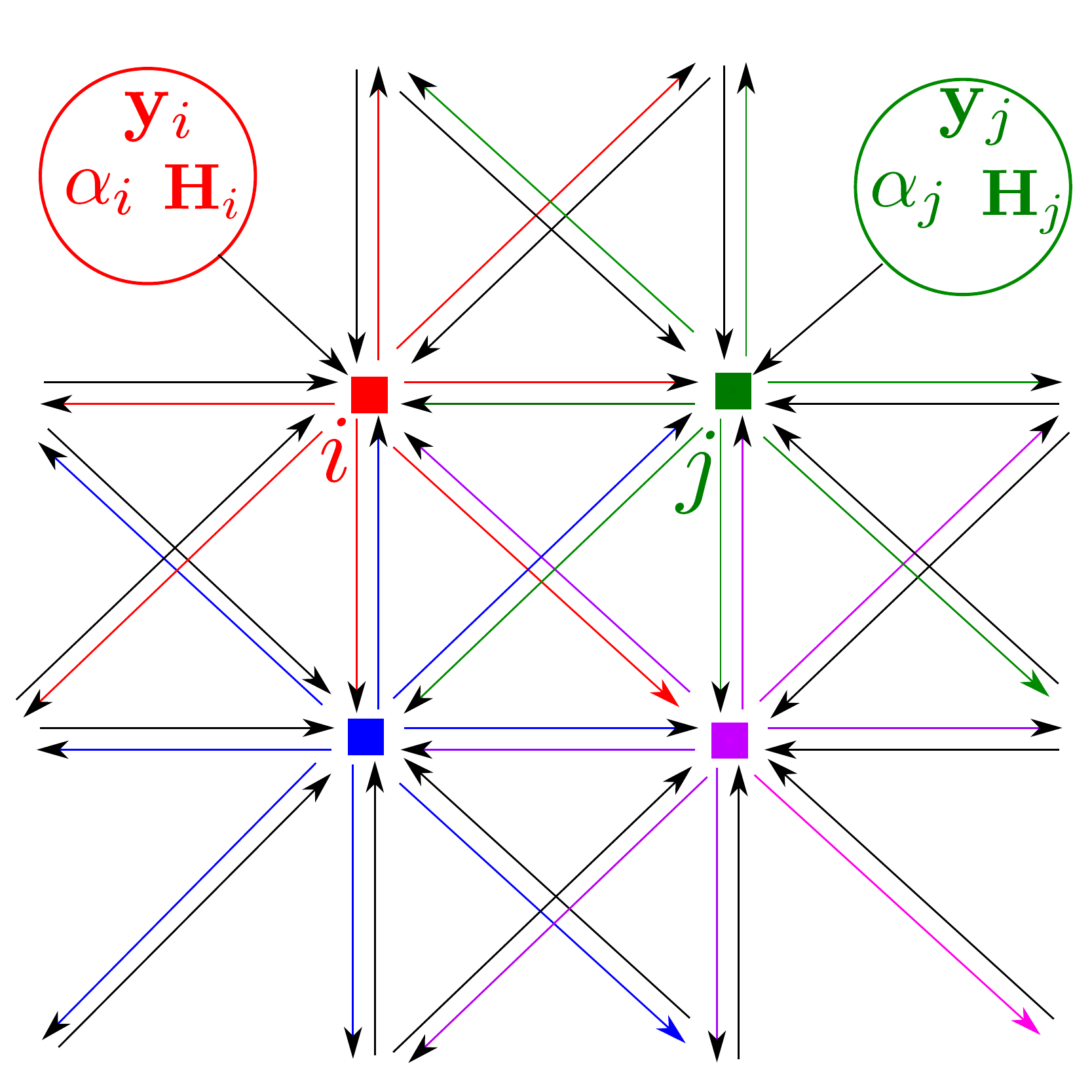}
	\caption{The communication network topology for our distributed image deblurring algorithm. Each square represents a node connected to its neighbors by bi-directional communication channel. Provided with blurred patches, and the corresponding local PSFs and the interpolation weights, the nodes run simultaneously the \textsc{Local-Solver}s, and exchange their local estimate of overlapping pixels to perform the consensus operation.}
	\label{fig:demonstration2}
\end{figure}

\subsection{Size of Image Blocks and the Extent of the Overlaps}
\label{subsec:patch_size_overlap_weights}
The size of observed image blocks and the extent of overlaps among them have impact upon the computational cost and the quality of final deblurred image. Thus, they must be chosen appropriately depending upon the situations: shift-invariant, piece-wise constant shift-variant, and smooth shift-variant blur. Without loss of generality, we assume that all local PSFs have the same size.

In the case of smooth shift-variant blur, the size of image blocks and the extent of overlaps depend upon how densely the grid points, thus the local PSFs, can be sampled within the field-of-view. Denser the grid points, better is the approximation of shift-variant blur in the image. If we consider $N$ equidistant grid points within the field-of-view, then the support of image blocks extend (in both horizontal and vertical directions) over three consecutive grid points, except for the image blocks at the boundaries of image. Thus, each block overlaps half of its size from top, bottom, left and right with its adjacent blocks, i.e., a pixel in the observed image appears into four adjacent blocks. As discussed previously, this extent of overlaps are necessary to impose smooth variation of blur among the image blocks. See the illustration in Fig.~\ref{fig:demonstration1}, which demonstrate for 1D signal the scheme of splitting the observed image into blocks, and the extent of overlaps among them. Similarly, Fig.~\ref{fig:shiftvariant_simulation_2}(a) shows $5\times 5$ grid points overlaid upon the observed image where the local PSFs are sampled, and Fig.~\ref{fig:shiftvariant_simulation_2}(c) shows resulting overlapping image blocks.

In the case of shift-invariant blur, again we can consider $N$ equidistant grid points within the field-of-view, and split the observed image into $N$ overlapping blocks. However, in this case, we do not need the same extent of overlaps among the blocks as in the case of smooth shift-variant blur, but a smaller extent of overlaps. In fact, it is sufficient to have extent of overlaps slightly larger than half the size of the local PSF. In this case, overlaps among the blocks are required not to impose smooth variation of blur (as the PSF same over whole field-of-view), but to minimize any discrepancies in the deblurred image due to approximation introduced in the regularization term. Similarly, in the case of piece-wise constant shift-variant blur, we can consider the same splitting scheme as in the case of shift-invariant blur for the same reason.

In the case of \emph{proposed} deblurring method, the maximum number blocks we can split the image is dependent upon the number of available nodes to process in parallel. In such a situation, the larger the number of available nodes, the larger the number of blocks we can chose, thus, smaller the size of blocks. From a computational point of view, the smaller the size of the blocks, the smaller the memory and computation time required by nodes to execute the \textsc{Local-Solver}\footnote{The size of vectors affects the efficiency of Fast Fourier Transform algorithm and other operations. Smaller vector fit well into the different levels of data-caches in processor so that it consumes lesser clock cycles to execute the operations compared to the situation when the vectors are larger to fit into the data-caches leading to many cache misses and eventually requiring several cycles to fetch them from main memory.}. Moreover, the smaller the size of blocks, the smaller is the extent of overlaps among them, thus, the smaller the amount of data to be exchanged among the nodes. With all these advantages from a computational point-of-view, we may intend to split the observed image into as many blocks as the number of available nodes, given that we are provided with as many local PSFs within the field of view. However, the size of the local PSF suggests a lower limit on the size of blocks. We should not select the size of blocks smaller than the size of the PSF, otherwise the local deblurring problem at each node will be highly underdetermined (more number of unknown pixels at boundaries of the blocks need to be estimated), which would eventually deteriorate the quality of the final deblurred image. Thus, to achieve better quality deblurred image, we should select the size of blocks at least twice the size of PSF.

\section{Numerical Experiments and Results}
\label{sec:numerical_experiments_results}
Our algorithm depicted in Algorithm~\ref{algo:proposed_algo} can be seen as general framework for distributed image deblurring. Depending upon the situation, we can derive a particular instance from it, e.g., shift-invariant or shift-variant deblurring, and depending upon the structures of data-fidelity and regularizer, we can select any efficient optimization method as the \textsc{Local-Solver}. In order to validate the \emph{proposed} deblurring method, we performed two numerical experiments, first, considering the simpler case of shift-invariant deblurring, and then more difficult case of the shift-variant deblurring. 

Moreover, to evaluate the performance of the \emph{proposed} deblurring method in terms of quality of deblurred image, we compared it with two other possible approaches. First obvious choice was \emph{centralized} deblurring as a reference method that solves the original problem (\ref{Eq:specific_img_deblurring_problem}). Since \emph{centralized} deblurring method would require a huge memory for a large image, thus we selected images of reasonable sizes for our experiments. We chose the na\"{\i}ve ``split, deblur independently, and merge'' approach as the second method. In oder to minimize incoherencies among the deblurred blocks obtained by the second method, we split the observed image into overlapping blocks, and blended the locally deblurred blocks into single image by weighted averaging of the overlapping regions by using the same 2D first-order interpolation weights as in \emph{proposed} method. Hereafter, we will refer to this approach as \emph{independent} deblurring method. The \emph{independent} deblurring method is not intended to solve correctly either of the optimization problems (\ref{Eq:specific_img_deblurring_problem}) or (\ref{Eq:genericpb}), but it is a straight forward way to deblur large images with minimal computational resources. Similarly, in the case of shift-variant image deblurring, we consider comparing our algorithm with the \emph{centralized} and \emph{independent} deblurring methods. For the \emph{centralized} deblurring method, we used the fast shift-variant blur operator based on first-order PSF interpolation as described in \cite{Denis2015}.

Let us first present the general settings of our experiments, and later more specific details. For the case of shift-invariant deblurring, we considered ``Lena'' image, which we resized to $1024\times 1024$ pixels, and extended its dynamic range linearly to have maximum intensity up to 6000 photons/pixels. We refer to this image as the reference image. The reference image was blurred with an Airy disk PSF (of size $201\times201$ pixels) formed due to a circular aperture of radius 6 pixels, and was corrupted with white Gaussian noise of variance $\sigma^2=400$ photons/pixels to obtain the observed image. To study the impact of factors discussed in Section \ref{subsec:patch_size_overlap_weights}, we considered different cases by varying number of blocks and the extent of overlaps among them. Figure~\ref{fig:shiftinvariant_simulation} shows an instance when observed image is split into $3 \times 3$ blocks with overlaps of $100 \times 100$ pixels among the adjacent blocks.

For the shift-variant deblurring case, we considered ``Barbara'' image, which was resized it to $1151 \times 1407$ pixels, and its dynamic range was extended linearly to have maximum intensity up to 6000 photons/pixels. As above, we call this image as reference image. We generated $9 \times 9$ normalized Gaussian PSFs each of size $201 \times 201$ pixels with the central PSF having full-width-half-maximum (FWHM) $= 3.5 \times 3.5$ pixels, and linearly increasing FWHM in the radial direction up to $16.5 \times 10.5$ pixels for the PSF at extreme corner of the reference image as depicted in Fig.~\ref{fig:shiftvariant_simulation_1}. These PSFs, somehow, mimic the shift-variant blur due to optical aberrations (coma) \cite{Mahajan2011}. We blurred the reference image with this shift-variant PSFs using shift-variant blur operator based on PSF interpolation described in \cite{Denis2015}. We obtained the final observed image by adding white Gaussian noise of variance $\sigma^2=400$ photons/pixels to the blurred image. As in shift-invariant deblurring case, we considered different scenarios by varying the number of blocks. Figure~\ref{fig:shiftvariant_simulation_2} shows an experimental setup when using only $5 \times 5$ gird of PSFs, i.e., when observed image split is into $5 \times 5$ overlapping blocks.

In both the cases, we deliberately selected low level noise in the observed image so that any incoherency or artifact arising in the deblurred image due to the approximations we made was not superseded by the strong regularization level required for low signal-to-noise ratio in observed image. Also, we selected sufficiently large size of PSFs so that they are not band-limited, which is the case in many real imaging systems.

\subsection{Choice of Regularization Term}
For our experiments, we selected the regularization function $\phi$ to be Huber loss and $\M{D}$ to be circular forward finite difference operator, so the regularizer is written as
\[
\phi(\M{D}_{i} \V{x}_{i}) = 
\begin{cases}
\frac{1}{2} \| \M{D}_{i} \V{x}_{i} \|^{2}_{2} 				& \quad \| \M{D}_{i} \V{x}_{i}  \|_{2} \leq \delta \\
\delta (\| \M{D}_{i} \V{x}_{i} \|_{2} - \frac{\delta}{2})	& \text{otherwise} \\
\end{cases}
\]
We chose this regularizer for two reasons: i) it is smooth (differentiable), which makes the functions $f_{i}$ in (\ref{Eq:genericpb}) smooth so that we can choose any fast optimization algorithm such as quasi-Newton methods (e.g., BFGS class methods \cite{Nocedal2006}) as the \textsc{Local-Solver}, and ii) it behaves in between Tikhonov and the total-variation regularization, depending upon the value of $\delta$. Thus, it is able to preserve the sharp structures in the images while avoiding stair-case artifacts in smoothly varying regions usually rendered by total-variation regularizer \cite{mugnier2004mistral}. However, our \emph{proposed} deblurring method is not restricted to any particular choice of regularizer; depending upon the application one can chose any regularizer and any fast optimization algorithm for local deblurring problem.

\subsection{Implementation Details} 
Since the optimization problems arising in all deblurring methods were smooth, thus we chose a quasi-Newton method called limited-memory variable metric with bound-constraint (VMLM-B)\footnote{An open source implementation of VMLM-B in C programming is available at  \url{https://github.com/emmt/OptimPack}.} \cite{Thiebaut2002}. VMLM-B is a variant of the standard limited-memory BFGS with bound-constraint (LBFGS-B) \cite{Zhu1995}, but with some subtle differences. VMLM-B does not require any manual tunning of parameters to achieve fast convergence (the only parameter step-length at each iteration of VMLM-B is estimated by line-search methods satisfying Wolfe's conditions). Thus, this left us with only a single parameter $\gamma$ to be tunned to achieve faster convergence of our algorithm. Moreover, using the same algorithm (VMLM-B) for solving the optimization problems in all three deblurring methods ensured a fair comparison among them in terms of quality of deblurred images and the computational expenses.

In all our experiments, we set $\rho^{(k)} = 1$. After few trials, we found that $\gamma = 0.001$ results in fast convergence of our algorithm. Since D-R splitting algorithm converges even when the $\Prox$ operations are carried out inexactly, 
thus to speed-up our algorithm, we allowed the \textsc{Local-Solver} to be less accurate during the initial iterations, and then more accurate as the iterations progress \cite[Section 3.4.4]{Boyd2011}. To do so, we allowed VMLM-B to perform 10 inner iterations at the beginning, and increased it by 10 at every next iteration of main loop. To speed-up further, we did warm-start of the VMLM-B at every iterations of our algorithm by supplying $\V{x}^{(k)}$ as an initial guess for next $\V{x}^{(k+1)}$ estimation. All the results presented below were obtained after 25 iterations of the proposed algorithm. It was observed that 25 iterations were generally sufficient, and more than 25 iterations did not bring any noticeable difference in the results. For the \emph{centralized} and \emph{independent} deblurring methods, we allowed VMLM-B to perform a maximum of 1000 iterations or until it satisfied its own stopping criterion based on gradient convergence or progress in the line-search method. It was noticed that, generally, 500 to 600 iterations were sufficient and further iterations did not bring any noticeable difference.

All the three deblurring methods including the \emph{centralized}, \emph{independent} and the \emph{proposed} were implemented in high-level dynamic programming language ``Julia''\footnote{\url{http://julialang.org/}} \cite{bezanson2012julia}. The \emph{proposed} distributed deblurring is implemented using Message Passing Interface (MPI) based on Open MPI library\footnote{\url{https://www.open-mpi.org/}}. The source codes for all the demonstrations shown in this paper will be freely available at \url{https://github.com/mouryarahul/ImageReconstruction}.

\subsection{Results and Discussion}
\label{subsec:results}
As mentioned above, we conducted two different experiments, one for shift-invariant, and another for smooth shift-variant blurred image. To compare the quality of the deblurred images obtained from all the three methods we considered two image quality metrics: signal-to-noise ratio (SNR), and Structural Similarity Index (SSIM) \cite{Wang2004}. In all the cases, we heuristically fixed the regularization parameter $\delta = 100$, and performed deblurring for different values of $\lambda$ in a sufficiently large range to see if one of the methods produces better quality image for certain range of $\lambda$ than others. The plots (SNR vs $\lambda$ and SSIM vs $\lambda$) in the  Fig.~\ref{fig:shiftinvariant_simulation}(i\textendash j) show the results obtained from shift-invariant deblurring when observed image was split into $3 \times 3$ blocks. To study the influence of size of blocks and extent of overlaps on the quality of deblurred image, we conducted several trails with different settings. The plots in the Fig.~\ref{fig:shiftinvariant_simulation}(k\textendash l) show the influence of extent of overlaps on the image quality. The plots shows that our algorithm performed slightly better than the \emph{centralized} deblurring in terms of both the SNR and SSIM. This could be due to the approximation introduced in the regularization term in the \emph{proposed} deblurring method, e.g., sum of the regularizer on the blocks may be favorable than the regularizer on whole for some images depending upon the contents in the image. Moreover, the \emph{proposed} deblurring may also be benefited by the explicit overlaps among the adjacent blocks; each node can result into slightly different values of the overlapping pixels, which eventually leads to better estimate of those pixels after averaging the estimates from different nodes. This is indicated by the fact that there is slight increase in both the SNR and SSIM with increase in the extent of overlaps as seen in the plots Fig.~\ref{fig:shiftinvariant_simulation}(k\textendash l). It also indicates that an overlap equal to half the size of PSF is sufficient, and a larger overlap did not produce much improvement, but, of course, it did increase the computational and communication cost. We also noticed that the SNR and SSIM of deblurred image obtained from \emph{proposed} deblurring is slightly less dependent upon the extent of overlaps compared to that obtained by \emph{independent} deblurring.

The plots in Fig.~\ref{fig:shiftvariant_SNRvsLambda}(a\textendash b) compare the results obtained by the three deblurring methods for the case of smooth
shift-variant blurred image. First, we noticed that the image quality obtained by all the three methods improves drastically with the increase in density of grid of PSFs sampled within field-of-view. The SNR and SSIM of deblurred image obtained using only $3 \times 3$ PSFs is significantly lower than the cases when $5 \times 5$, $6 \times 6$ and $8 \times 8$ grids of PSFs are used. This is due to the fact that the observed image was simulated using a finer $9 \times 9$ grids of PSFs, and the coarser grids of PSFs are less accurate in capturing smooth variation of blur than a finer grid of PSFs. We also noticed that unlike the shift-invariant deblurring case, the values of SNR and SSIM obtained by the \emph{proposed} deblurring is slightly lower than that obtained by the \emph{centralize} deblurring. This could be due to the fact that both the \emph{centralized} and \emph{proposed} deblurring benefits from the explicit overlaps among the block, but some information is always lost at the boundaries of the deblurred blocks in the case of \emph{proposed} deblurring, which is not the case for \emph{centralized} deblurring. Also, there must be some influence of the approximation introduced into the regularization term for the case of \emph{proposed} deblurring. We observed from the plots in Fig.~\ref{fig:shiftinvariant_simulation}(i\textendash j) and Fig.~\ref{fig:shiftvariant_SNRvsLambda}(a\textendash b) that three methods do not attained the highest SNR or SSIM at the same value of the regularization parameter $\lambda$.

As expected, we observed that the na\"{\i}ve \emph{independent} deblurring method performed significantly lower than the other two methods. As pointed out above, this is due to the fact that the method is not intended to solve correctly either the original problem (\ref{Eq:specific_img_deblurring_problem}) or the distributed formulation (\ref{Eq:genericpb}), but it is the crudest and computationally cheapest way to perform image deblurring by splitting image into pieces. 

In our experiments, we observed that the first step, i.e., the \textsc{Local-Solver}, of the \emph{proposed} deblurring algorithm is the one which consumed the significant computation time among the three steps; \textsc{Local-Solver} took 600 to 800 times more computation time than the consensus step (including communication time among the nodes). This should be true, in general, for many other local deblurring algorithms devised using different combination of data-fidelity and regularization terms depending upon the applications. Thus, the \emph{proposed} deblurring algorithm is efficient in the sense that it is computation intensive rather than being communication intensive. 

For the small or moderate size of images as we considered in our experiments, the \emph{proposed} deblurring algorithm is computationally more expensive than the \emph{centralized} and \emph{independent} deblurring methods; it takes at least 10 to 15 times more computation time than the latter ones. However, in the case of extremely large images for which \emph{centralized} deblurring is practically not feasible, computational expenses of \emph{proposed} deblurring is justified by the better quality of deblurred image, which cannot be achieved by the computationally cheaper \emph{independent} deblurring. Some more simulations performed with different set of images and PSFs are presented in Appendix \ref{sec:appendix}, and the results suggest similar conclusions as above.


\begin{figure*}
	\centering
	\subfloat[{Reference image (size = $1024 \times 1024$ pixels)}]{\includegraphics[width=0.24\linewidth]{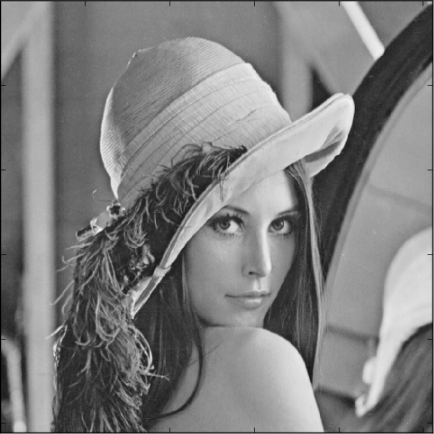}} \;
	\subfloat[Airy disk PSFs (size = $201 \times 201$ pixels) due to circular aperture of radius $6$ pixels]{\includegraphics[width=0.24\linewidth]{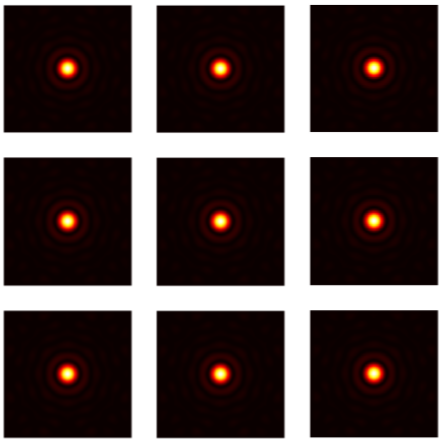}} \;
	\subfloat[{Observed image (SNR= 9.0364 dB, SSIM = 0.7594)}]{\includegraphics[width=0.24\linewidth]{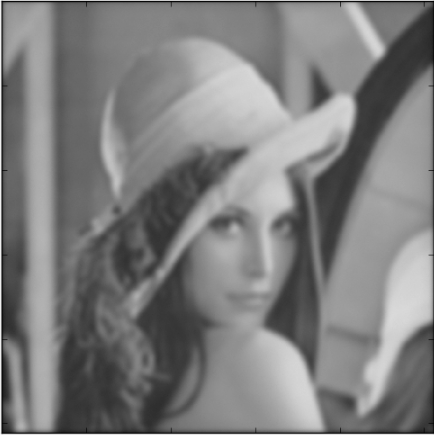}} \;
	\subfloat[{Overlapping observed patches}]{\includegraphics[width=0.24\linewidth]{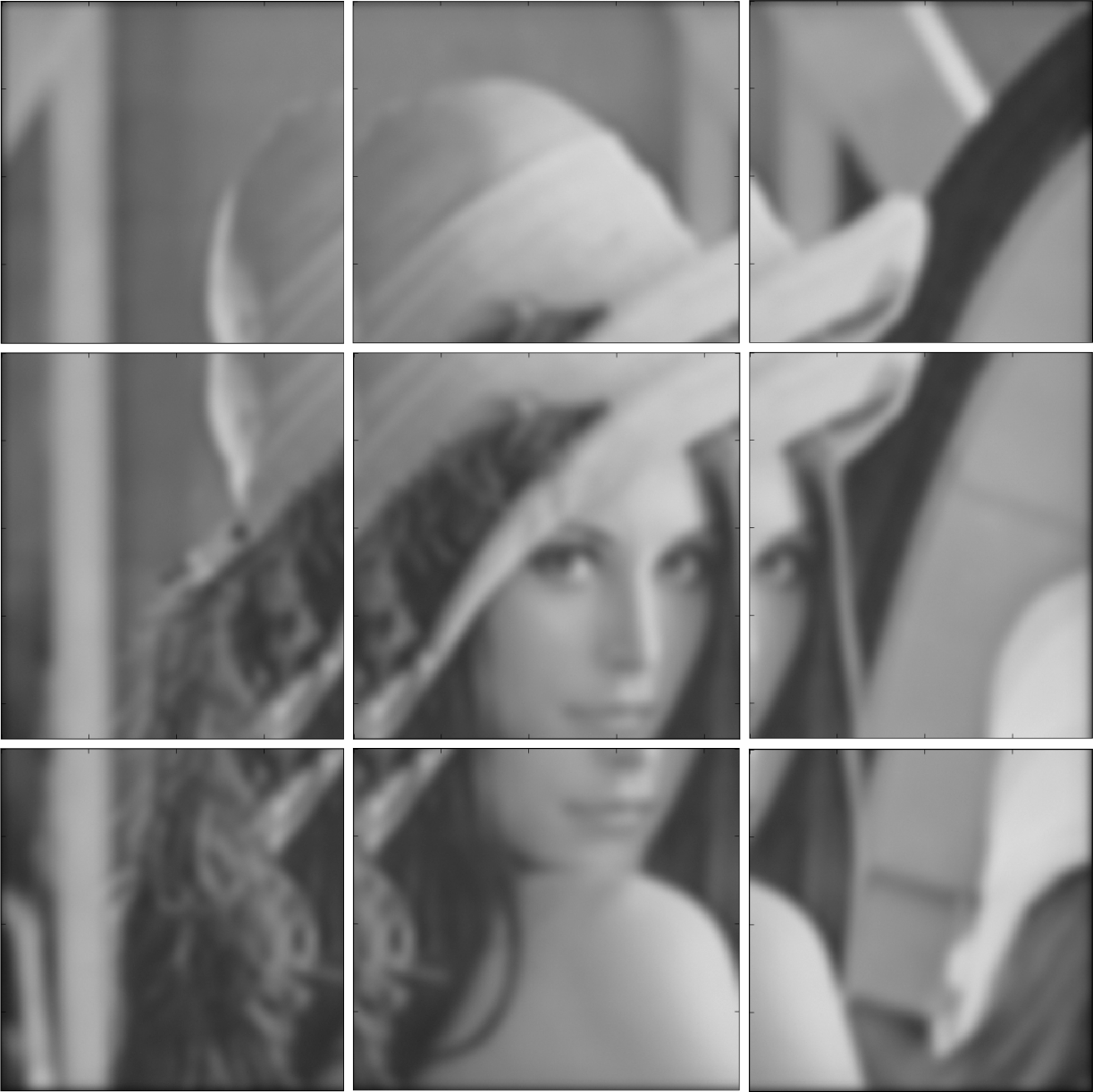}} \\
	\subfloat[{$3 \times 3$ blocks of interpolation weights with brightest pixel equal to 1 and the darkest pixel equal to 0.}]{\includegraphics[width=0.24\linewidth]{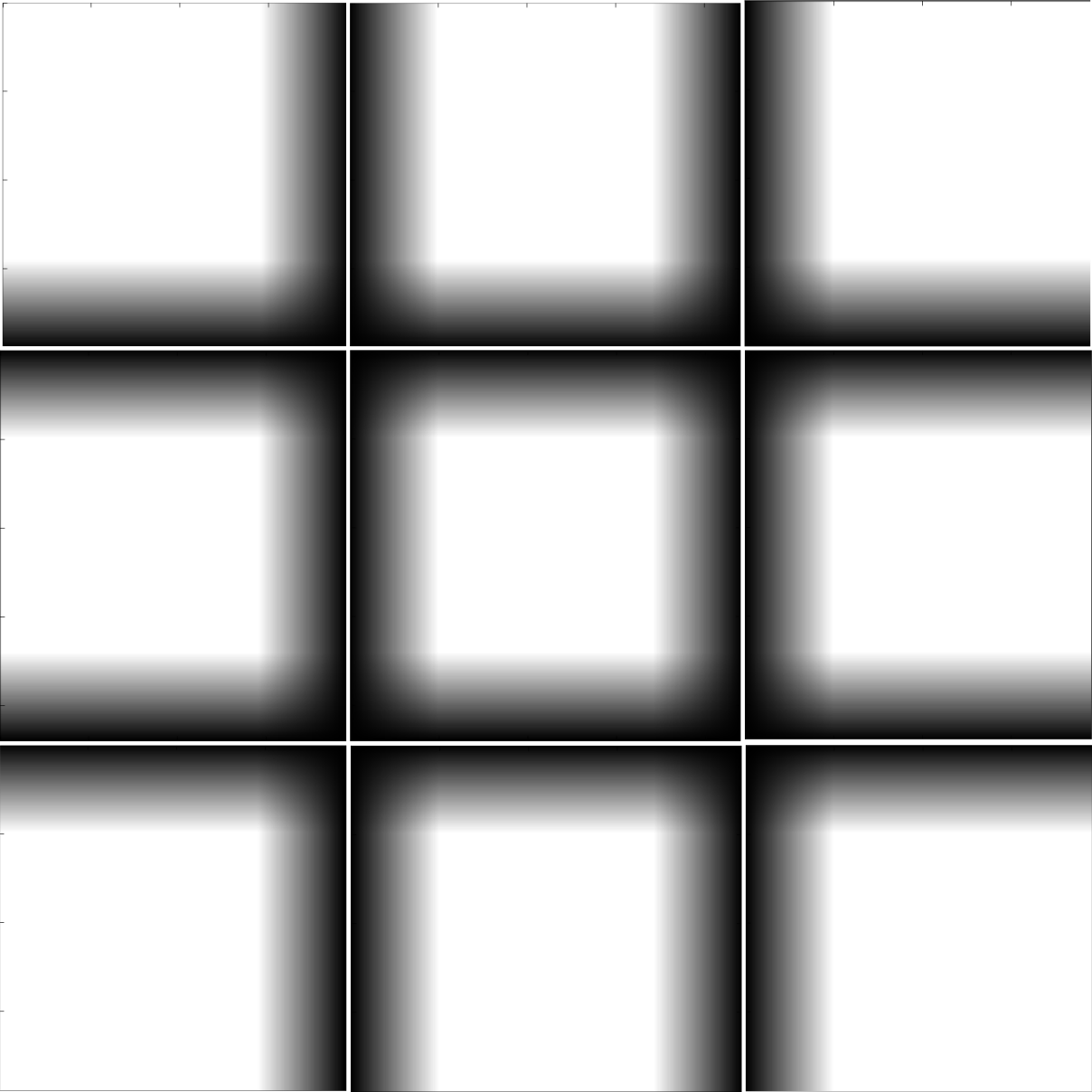}} \;
	\subfloat[{ Image obtained by \emph{centralized} deblurring method (SNR = 14.3818 dB, SSIM = 0.8166 at $\lambda = 0.001$)}] {\includegraphics[width=0.24\linewidth]{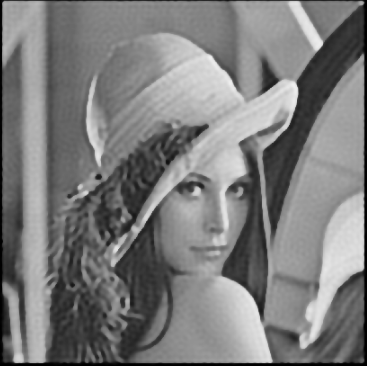}} \;
	\subfloat[{Image obtained by \emph{independent} deblurring method (SNR = 14.3402 dB, SSIM = 0.8162 at $\lambda = 0.001$)}] {\includegraphics[width=0.24\linewidth]{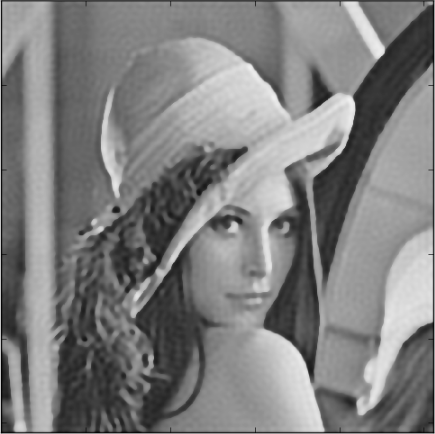}} \;
	\subfloat[{Image obtained by \emph{proposed} deblurring method (SNR = \textbf{14.5931} dB, SSIM = \textbf{0.8188} at $\lambda = 0.002$)}] {\includegraphics[width=0.24\linewidth]{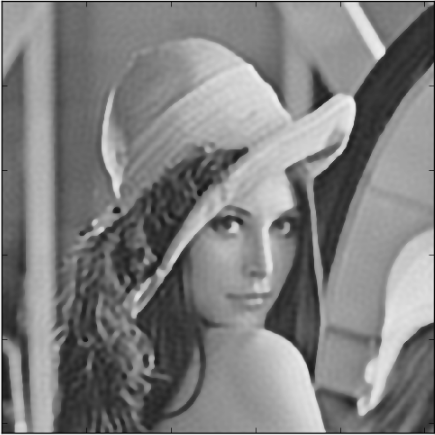}} \\
	\subfloat[{SNR vs $\lambda$}] {\includegraphics[trim = 0mm 0mm 20mm 16mm, clip,width=0.32\linewidth]{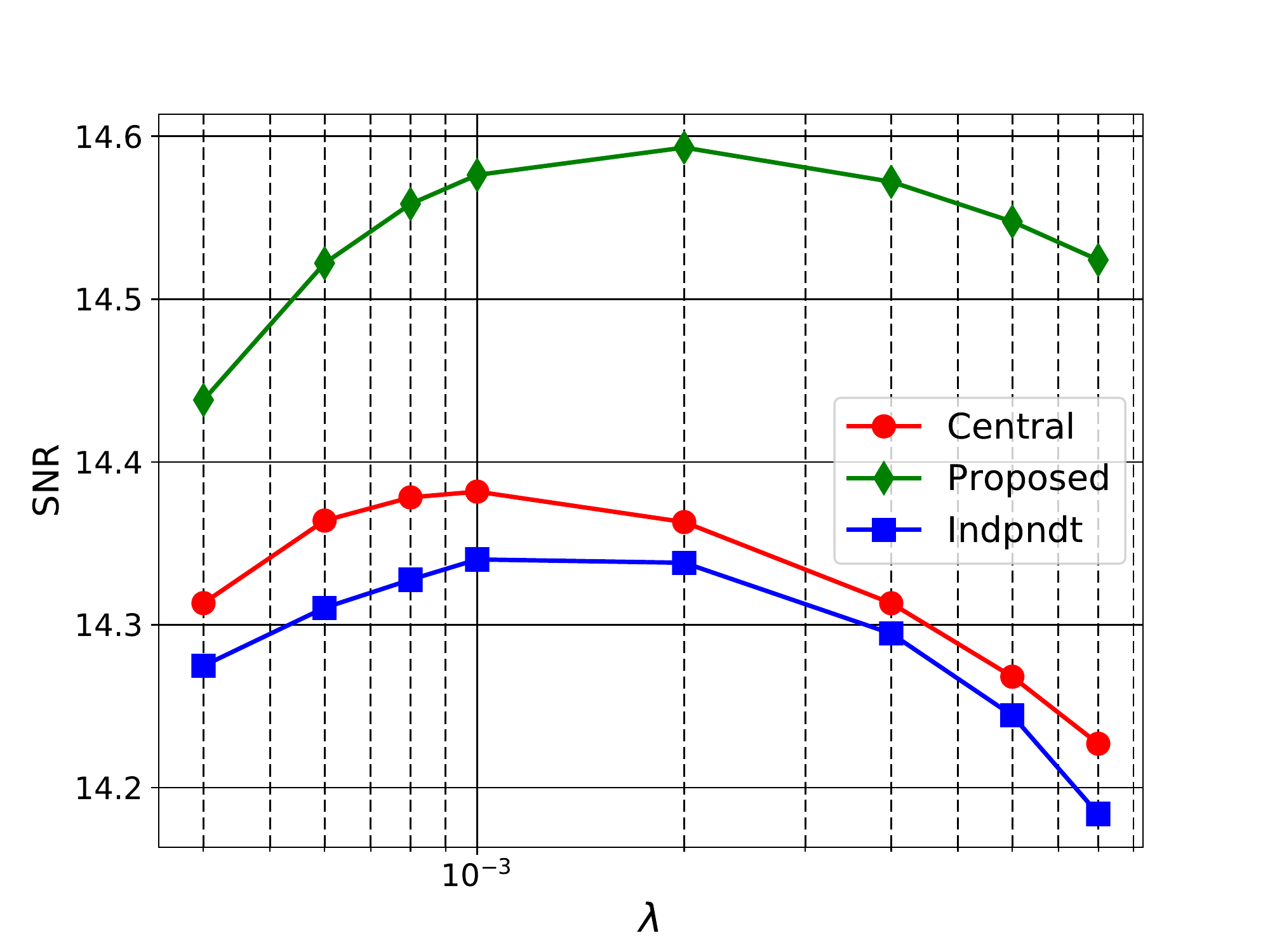}} \;
	\subfloat[{SSIM vs $\lambda$}] {\includegraphics[trim = 0mm 0mm 20mm 16mm, clip,width=0.32\linewidth]{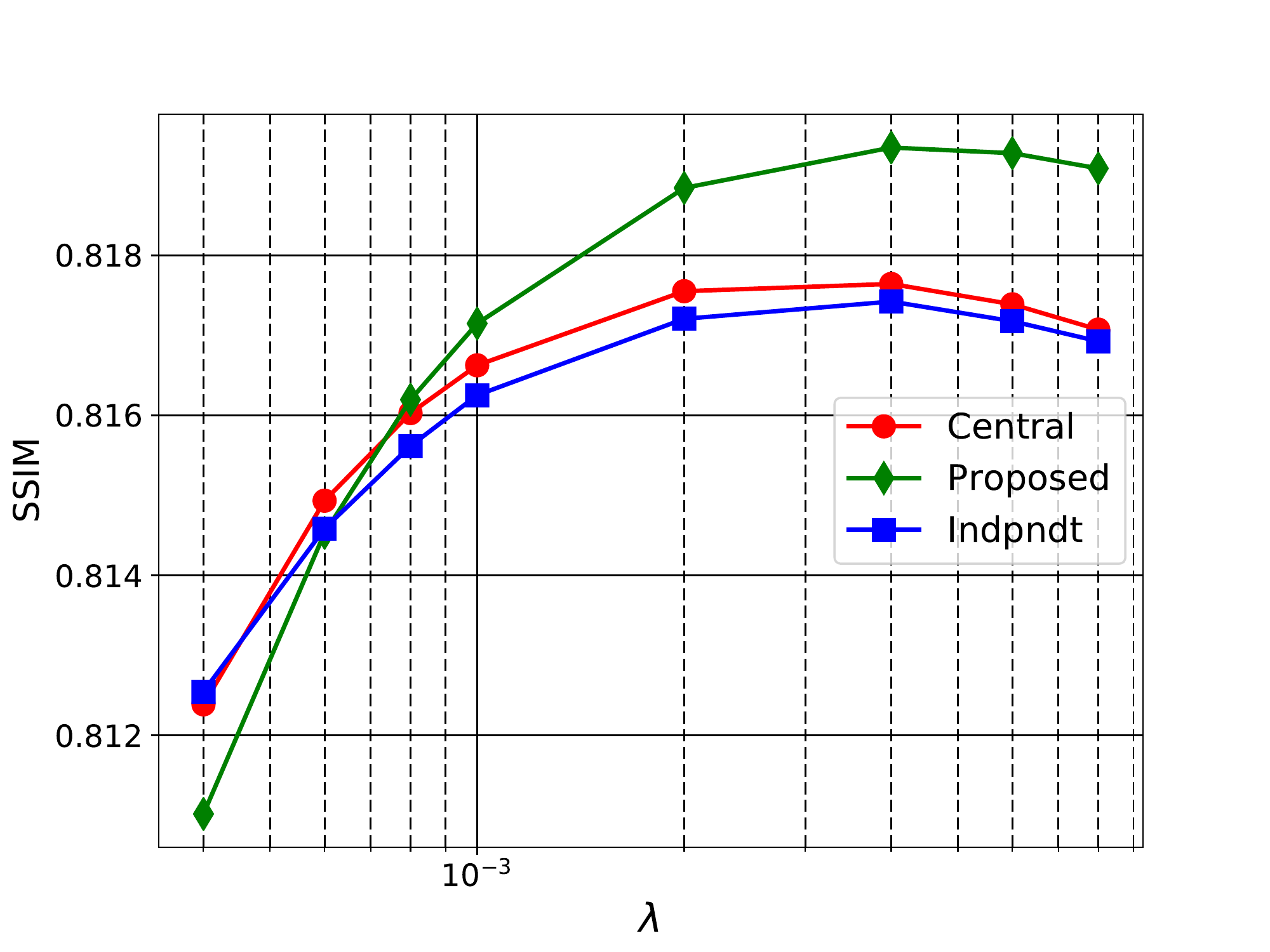}}\\
	\subfloat[{SNR vs Overlap}] {\includegraphics[trim = 0mm 0mm 20mm 16mm, clip,width=0.32\linewidth]{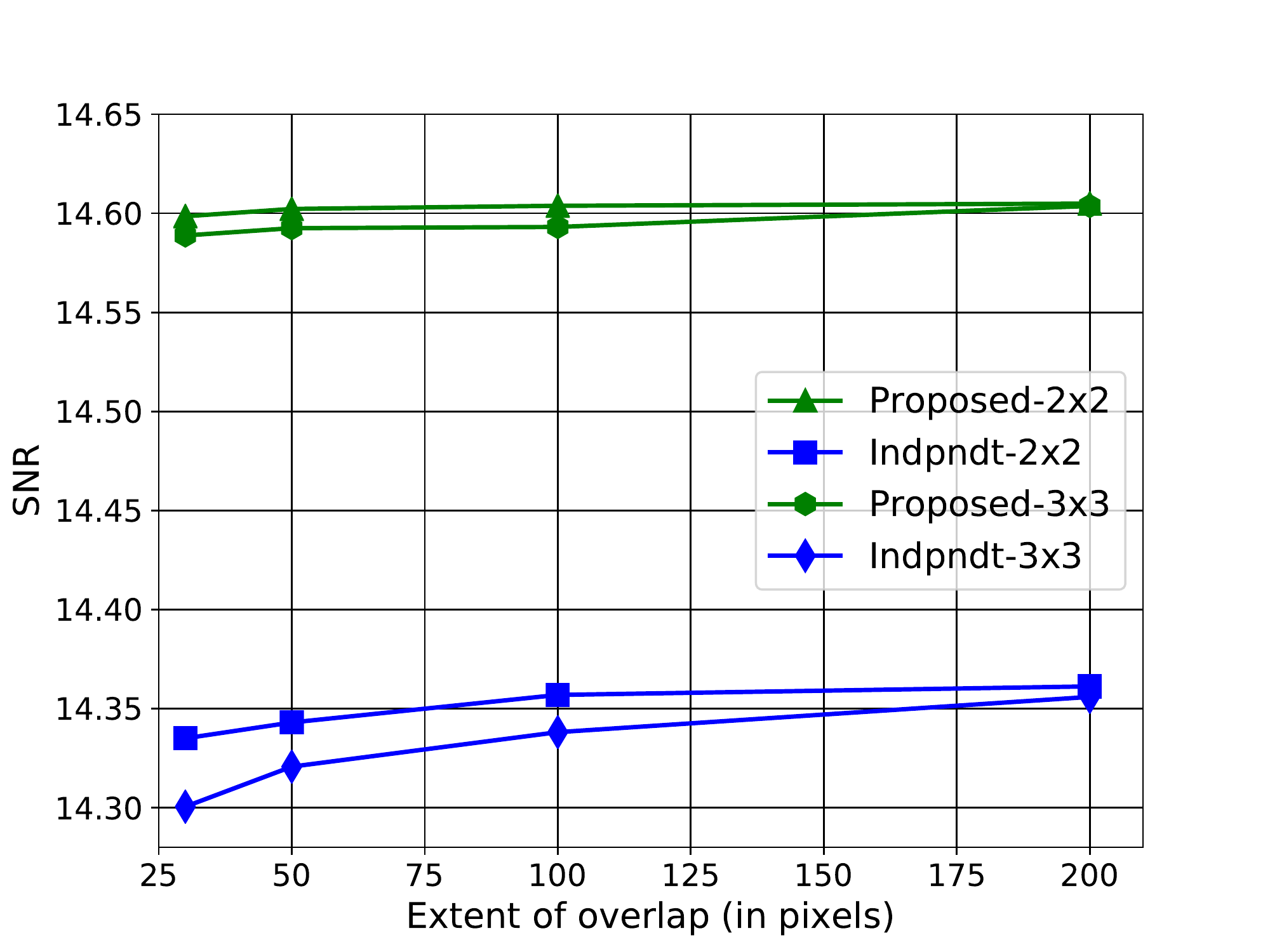}} \;
	\subfloat[{SSIM vs Overlap}] {\includegraphics[trim = 0mm 0mm 20mm 16mm, clip,width=0.32\linewidth]{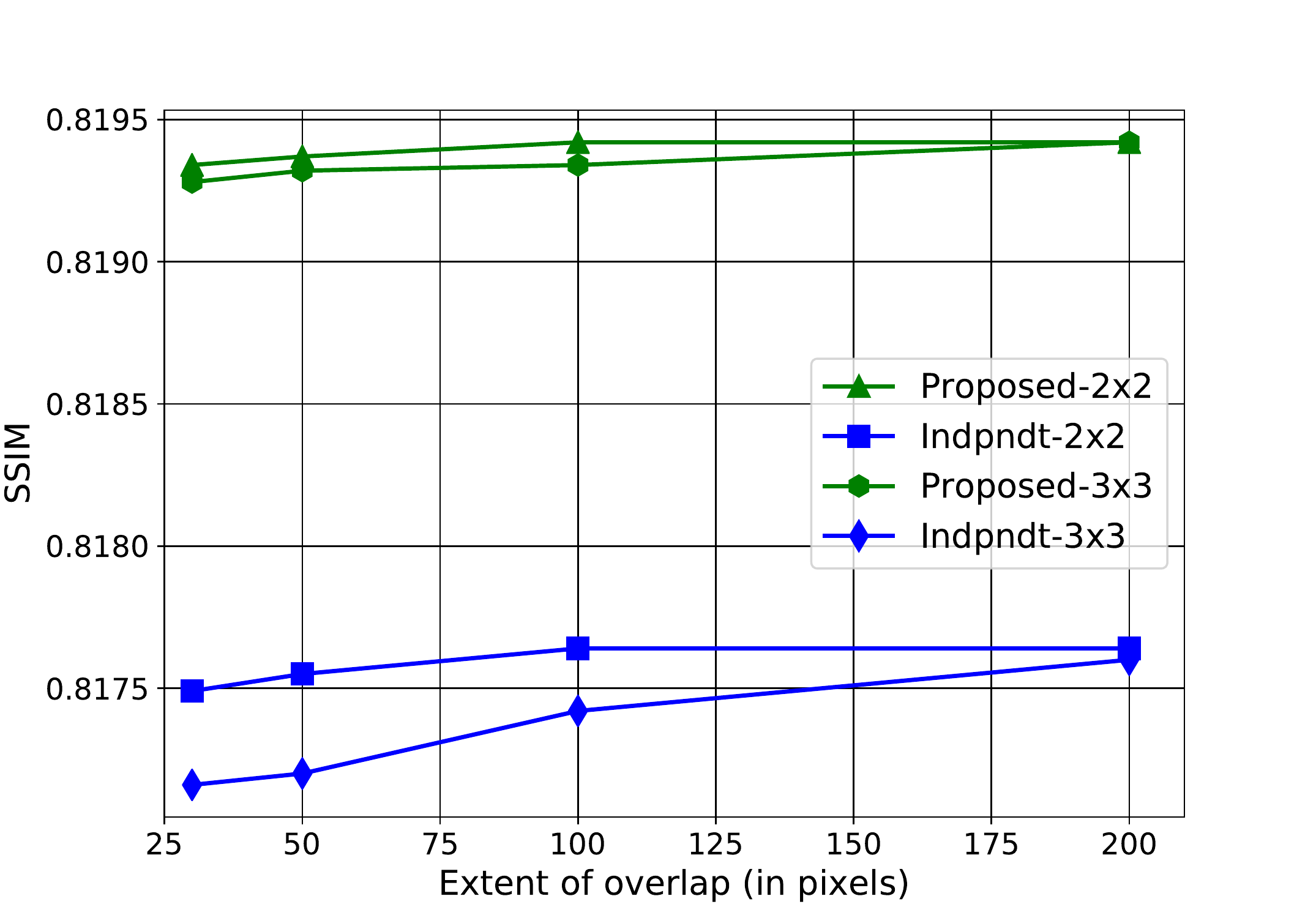}}
	\caption{Experiment 1: experimental setup and results from shift-invariant deblurring of ``Lena'' image. The observed image (c) is obtained by blurring the reference image (a) with shift-invariant PSF (b), and then corrupting with white Gaussian noise of variance $\sigma^2 = 400$. Image blocks (d) are obtained by splitting observed image into $3 \times 3$ blocks with overlaps of $100 \times 100$ pixels among them. The 2D first-order interpolation weights (e) are of the same size as observed blocks. Plots (i\textendash j) show the image quality of deblurred images obtained for different strengths of regularization. The legends ``\textcolor{red}{Central}'', ``\textcolor{green}{Proposed}'', and ``\textcolor{blue}{Indpndt}'' represent results from the \emph{centralized}, the \emph{proposed} and the \emph{independent} deblurring methods, respectively. Plots (k\textendash l) show impact of extent of overlap on the image quality of deblurred images. The legends ``\textcolor{green}{Proposed-2x2}'' and ``\textcolor{blue}{Indpndt-2x2}'' represent the \emph{proposed} and \emph{independent} deblurring methods for the case when the image is split into $2\times 2$ blocks. Similarly, the other two legends represent the case when the image is split into $3 \times 3$ blocks.}
	\label{fig:shiftinvariant_simulation}
\end{figure*}

\begin{figure}
	\centering
	\subfloat[{Reference image ($1151 \times 1407$ pixels) with $9 \times 9$ grid points (overlaid in green) where PSFs are sampled.}]{\includegraphics[trim = 0mm 0mm 0mm 0mm, clip,width=0.85\linewidth]{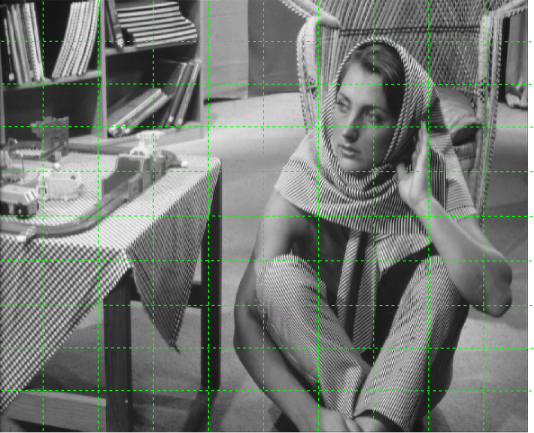}} \\
	\subfloat[Shift-variant PSFs (each of size $201 \times 201$ pixels) generated at the $9 \times 9$ grid points in (a).]{\includegraphics[trim = 0mm 0mm 0mm 0mm, clip,width=0.85\linewidth]{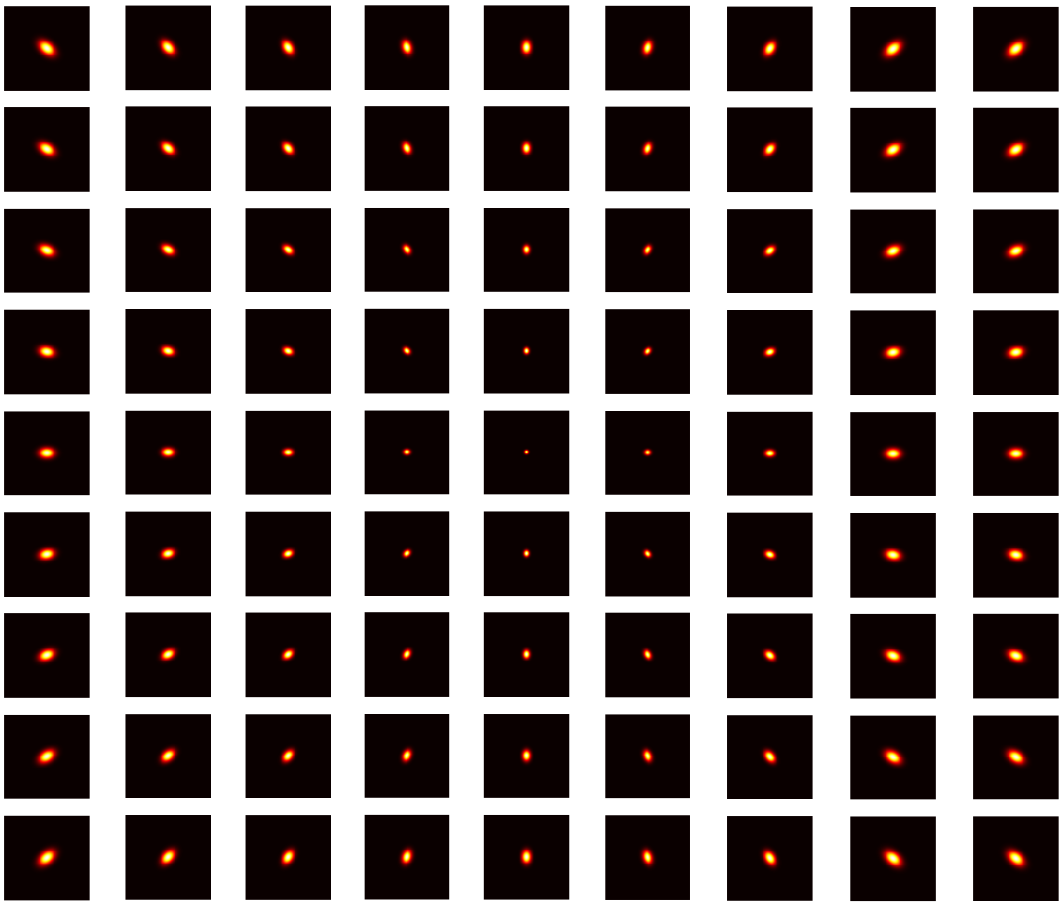}}\\	
	\subfloat[{Blurred and Noisy (observed) image}]{\includegraphics[trim = 30mm 15mm 25mm 15mm, clip,width=0.85\linewidth]{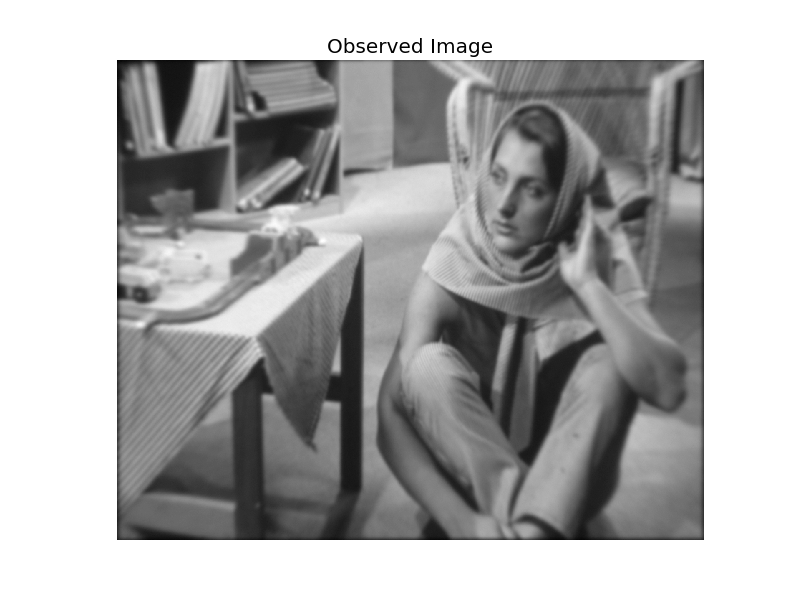}}
	\caption{Experiment 2: experimental setup for shift-variant deblurring of ``Barbara'' image. The grid of PSFs (b) contains normalized Gaussian PSFs with central PSF having FWHM = $3.5 \times 3.5$ pixels, and linearly increased up to FWHM = $16.5 \times 10.5$ pixels for the extreme corner PSF. The observed image (c) is obtained by blurring the reference image (a) with the shift-variant PSFs, and then corrupting white Gaussian noise of variance $\sigma^2 = 400$ photons/pixels.}
	\label{fig:shiftvariant_simulation_1}
\end{figure}

\begin{figure}
	\centering
	\subfloat[$5 \times 5$ grid points overlaid upon the observed image Fig. \ref{fig:shiftvariant_simulation_1}(c) where the PSFs are sampled.] {\includegraphics[trim = 0mm 0mm 0mm 0mm, clip,width=0.85\linewidth]{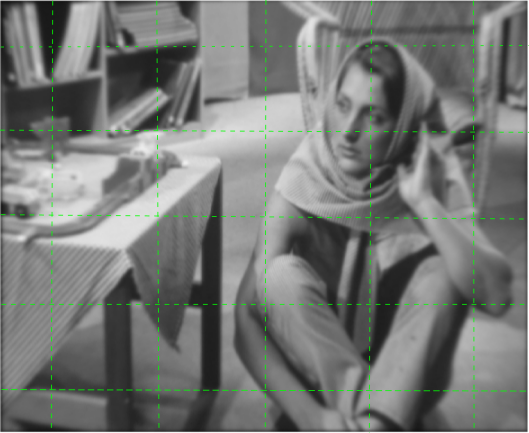}} \\
	\subfloat[Shift-variant PSFs (each of size $201 \times 201$ pixels) sampled at $5\times 5$ grid points shown in (a).] {\includegraphics[trim = 0mm 0mm 0mm 0mm, clip,width=0.85\linewidth]{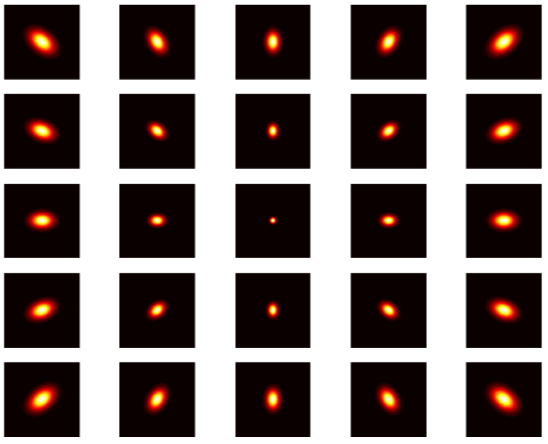}} \\
	\subfloat[$5 \times 5$ overlapping observed blocks obtained after splitting image in (a).] {\includegraphics[trim = 0mm 0mm 0mm 0mm, clip,width=0.85\linewidth]{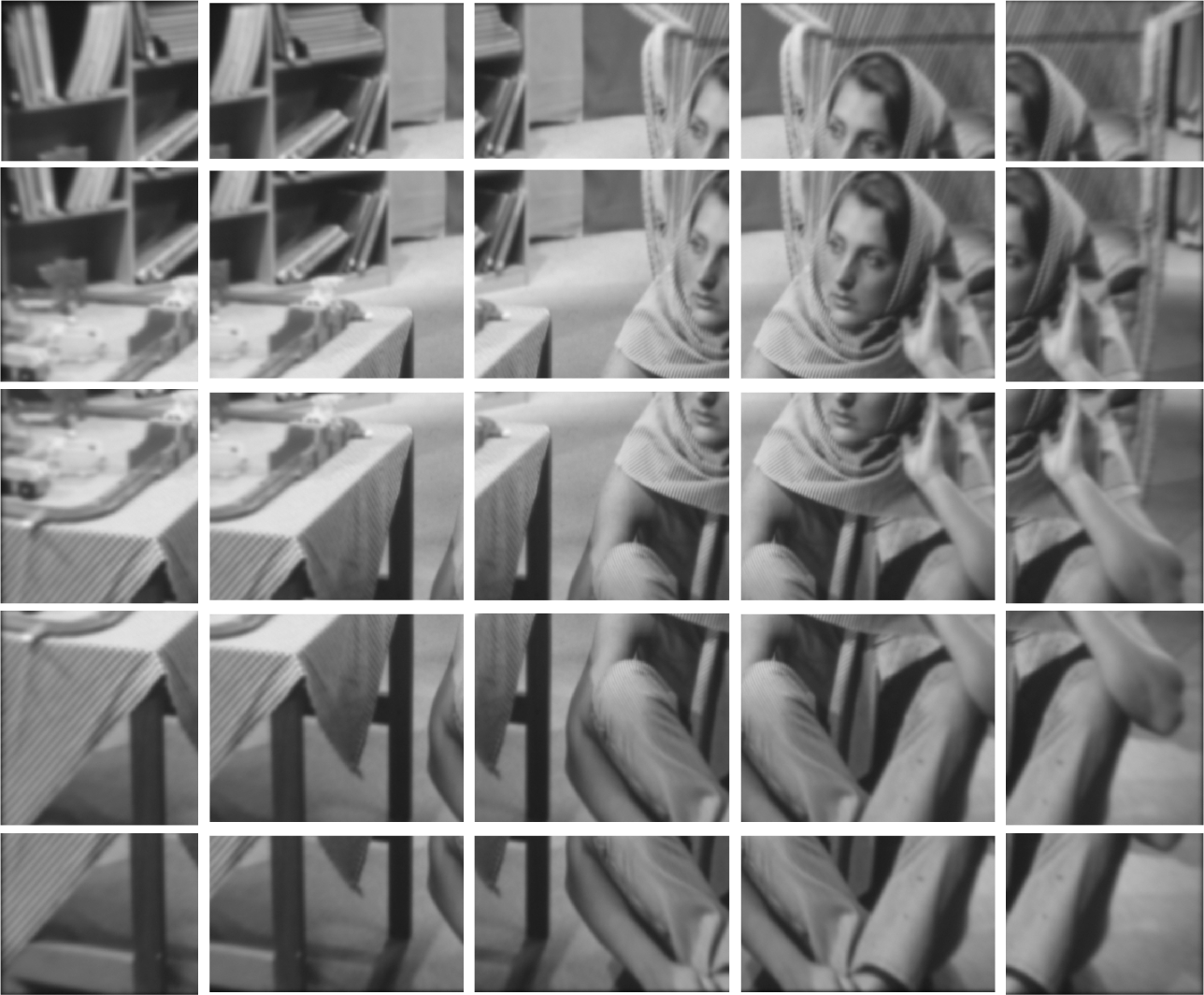}}
	\caption{Experiment 2: experimental setup for shift-variant deblurring when using $5 \times 5$ grid of PSFs, i.e., observed image is split into $5\times 5$ overlapping blocks.}
	\label{fig:shiftvariant_simulation_2}
\end{figure}

\begin{figure}
	\centering
	\subfloat[SNR vs $\lambda$]{\includegraphics[trim = 0mm 0mm 0mm 0mm, clip,width=0.95\linewidth]{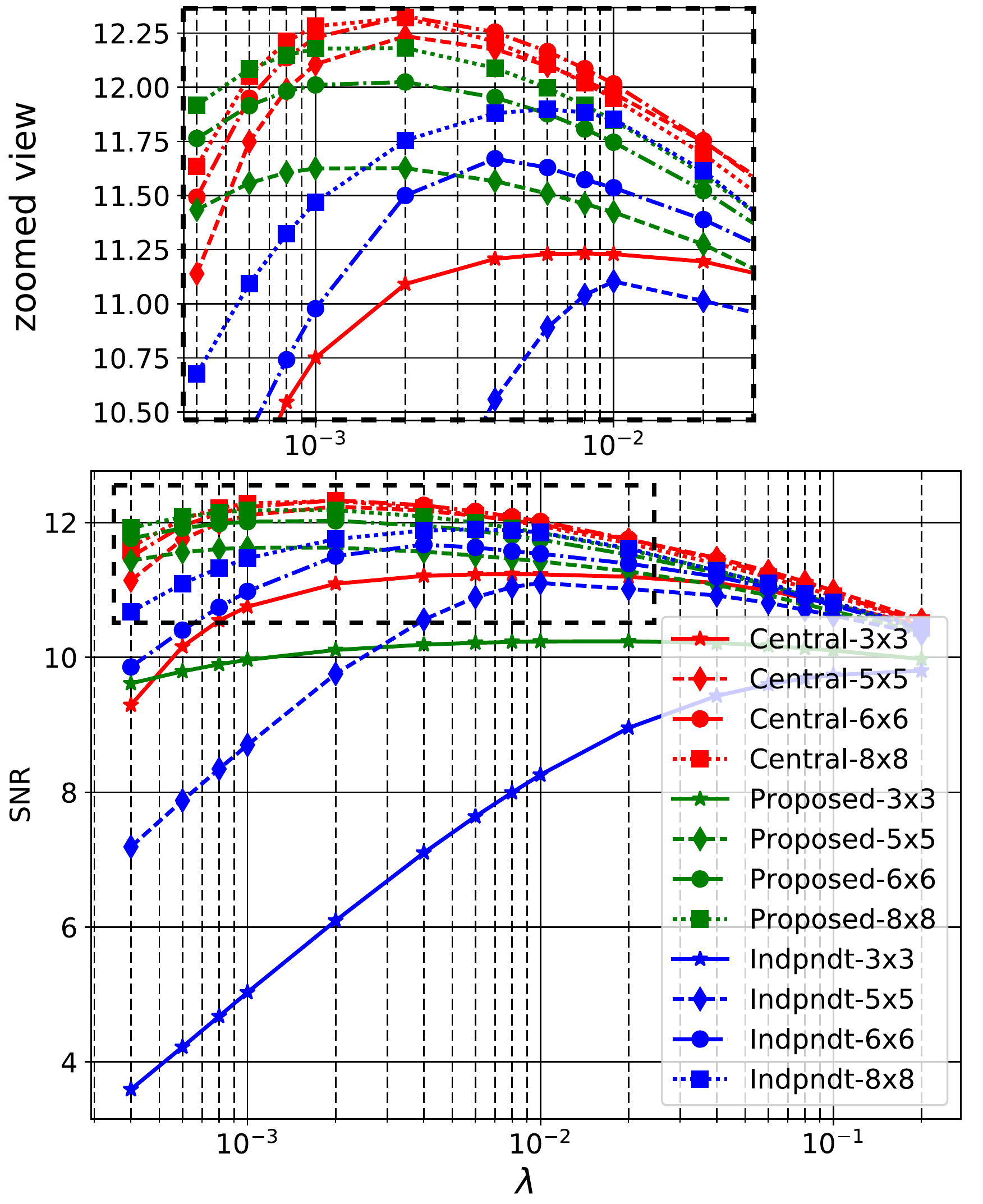}} \\
	\subfloat[SSIM vs $\lambda$]{\includegraphics[trim = 0mm 0mm 0mm 0mm, clip,width=0.95\linewidth]{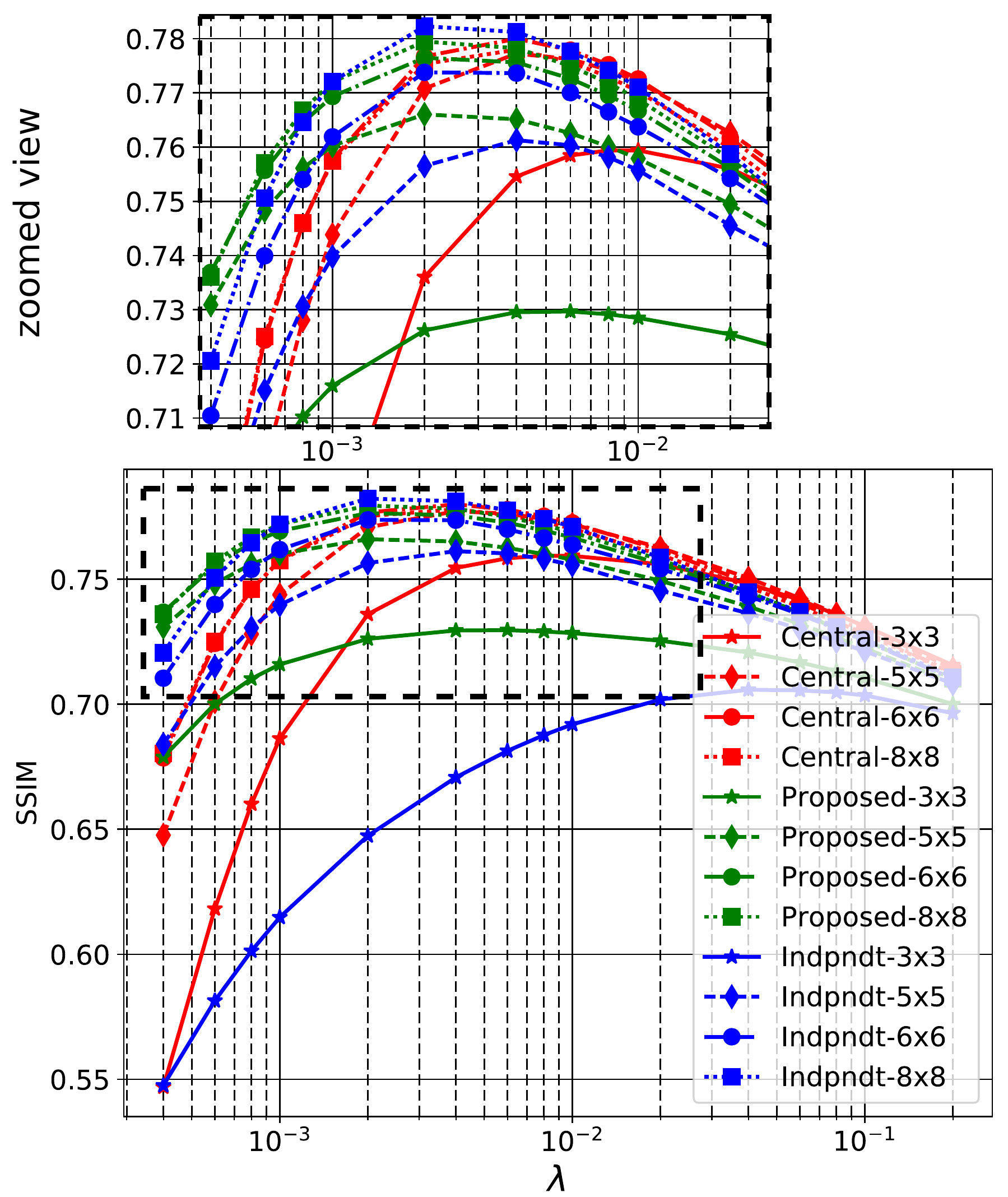}}
	\caption{Experiment 2: results from shift-variant image deblurring comparing the image quality (in terms of SNR and SSIM) obtained by the three different deblurring methods for different strength of regularization. The legends ``\textcolor{red}{Central3x3}'', ``\textcolor{green}{Proposed-3x3}'', ``\textcolor{blue}{Indpndt-3x3}'' denotes the results from the \emph{centralized}, the \emph{proposed}, and \emph{independent} deblurring methods, respectively, when using only $3\times 3$ grid of PSFs. Similarly, other legends denotes for the results obtained when using $5\times 5$, $6 \times 6$ and $8 \times 8$ grid of PSFs sampled in the field-of-view.}
	\label{fig:shiftvariant_SNRvsLambda}
\end{figure}

\begin{figure}
	\centering
	\subfloat[Estimated by \emph{centralized} deblurring (SNR = \textbf{12.3278} dB, SSIM = \textbf{0.7767} at $\lambda = 0.002$).] {\includegraphics[trim = 30mm 10mm 20mm 15mm, clip,width=0.85\linewidth]{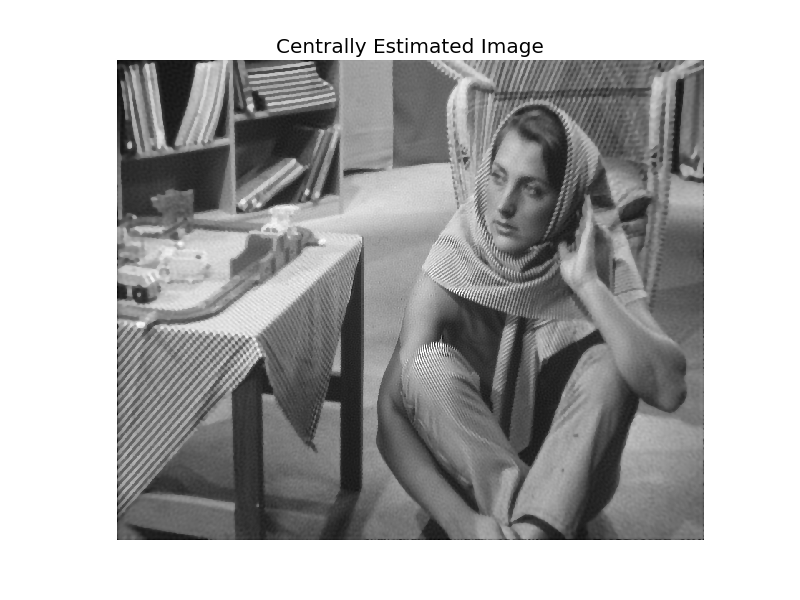}} \\
	\subfloat[Estimated by \emph{independent} deblurring (SNR = 11.6696 dB, SSIM = 0.7736 at $\lambda = 0.004$)] {\includegraphics[trim = 30mm 10mm 20mm 15mm, clip,width=0.85\linewidth]{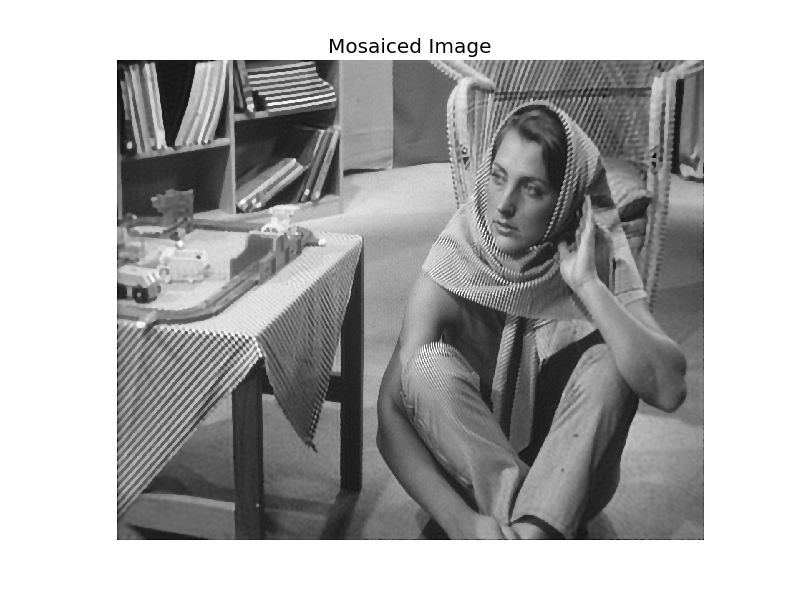}} \\
	\subfloat[Estimated by \emph{proposed} deblurring (SNR = {12.0239} dB, SSIM = {0.7764} at $\lambda = 0.002$)] {\includegraphics[trim = 30mm 10mm 20mm 15mm, clip,width=0.85\linewidth]{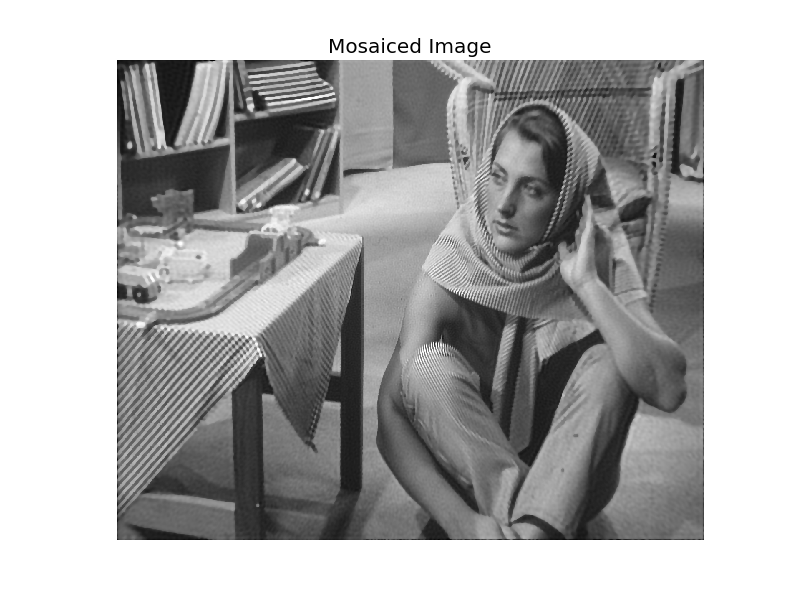}}
	\caption{Experiment 2: deblurred images obtained by the three methods when using $6\times 6$ grid of PSFs.}
	\label{fig:shiftvariant_deblur_result_6x6}
\end{figure}

\section{Conclusion}
\label{sec:conclusion}
In this paper we have proposed a distributed image deblurring algorithm for large images for which the existing \emph{centralized} deblurring methods are practically inapplicable due to requirement of huge physical memory on a single system. The proposed algorithm is rather a generic framework for distributed image deblurring for it can handle different imaging situations such as deblurring a single large image suffering from shift-invariant blur, a wide field-of-view captured into multiple narrow images by different imaging systems with slightly different PSFs, and a large image of a wide field-of-view suffering from smoothly varying blur captured by a single imaging system. Depending upon the application, one can easily adapt it to include different data-fidelity and regularization terms, and then select any fast optimization algorithm to solve the local deblurring problem at the different nodes of a distributed computing system. Our algorithm is efficient in the sense that it is computation intensive rather than being communication intensive. We showed by experimental results on simulated observed images that the \emph{proposed} deblurring algorithm produce almost similar or little lower quality (measured in term of SNR and SSIM) of deblurred images than that obtained by \emph{centralized} deblurring. But, this small compromise in the quality of deblurred image is trade-off by the cost effectiveness of our distributed approach for large images, which is practically not feasible for the \emph{centralized} deblurring methods. Moreover, we compared the \emph{proposed} deblurring to a na\"{\i}ve and computationally cheaper \emph{independent} deblurring, and showed the latter always performed significantly lower than the former. Thus, when high accuracy is desirable, e.g., in astronomical application, the \emph{proposed} deblurring should be preferred over \emph{independent} deblurring method, of course at expense of extra computational cost.

In this paper, we considered nonblind image deblurring, i.e., the PSFs were known a priori, however, in real imaging scenario calibrating PSFs accurately is a tedious and challenging task. A more practical way would be to follow a blind image deblurring approach which is able estimate the PSFs from the observed image(s), and recover the single crisp image. Thus, the next perspective step would be to extend the \emph{proposed} algorithm toward distributed blind image deblurring method which should be able to estimate distributively the PSFs imposing certain regularity among them, and eventually estimate the unknown crisp image.

\bibliographystyle{IEEEtran}
\bibliography{references}

\appendix[Some more simulation results]
\label{sec:appendix}
We repeated the above experiments for different set of images and PSFs, what we referred to as Experiment 3 and 4. This time we chose PSFs that caused less blurriness in the observed images than the previous experiments. For shift-invariant image deblurring, we chose ``Barbara'' image, and a Airy disk PSF of size $201 \times 201$ pixels formed due to a circular aperture of radius $9.5$ pixels. For smooth shift-variant image deblurring, we chose ``Pentagon'' image of size $2048 \times 2048$ pixels, and a grid of shift-variant normalized Gaussian PSFs of size $201 \times 201$ pixels with central PSF having FWHM = $3.5 \times 3.5$ pixels and linearly increasing FWHM in the radial direction up to $8.5 \times 6.5$ pixels for the PSF at extreme corner of the reference image. In both the cases, they dynamic range of the reference image was extended linearly up to $6000$ photons/pixels and the noise in the observed images was white Gaussian with variance $\sigma^2 = 400$ photons/pixels. The results from both new experiments are shown in Fig.~\ref{fig:expmnt3_shift_invariant_deblurring} and Fig.~\ref{fig:expmnt4_shift_variant_deblurring}. Again, we have very similar conclusions from this repeated experiments as from the previous experiments.

\begin{figure*}
	\centering
	\subfloat[SNR vs $\lambda$]
	{\includegraphics[trim = 2mm 0mm 20mm 16mm, clip, width=0.48\linewidth]{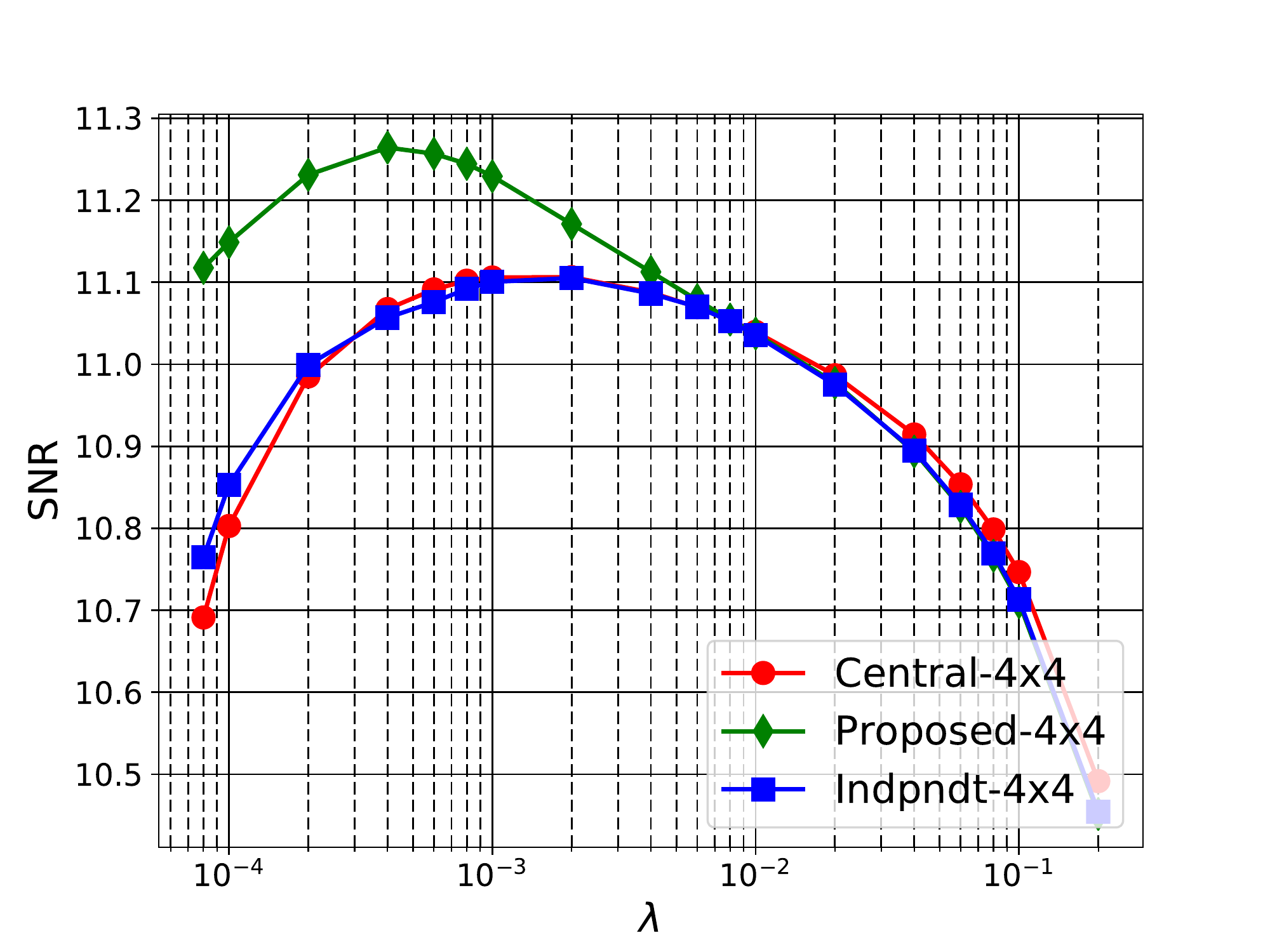}}
	\subfloat[SSIM vs $\lambda$]
	{\includegraphics[trim = 2mm 0mm 20mm 16mm, clip, width=0.48\linewidth]{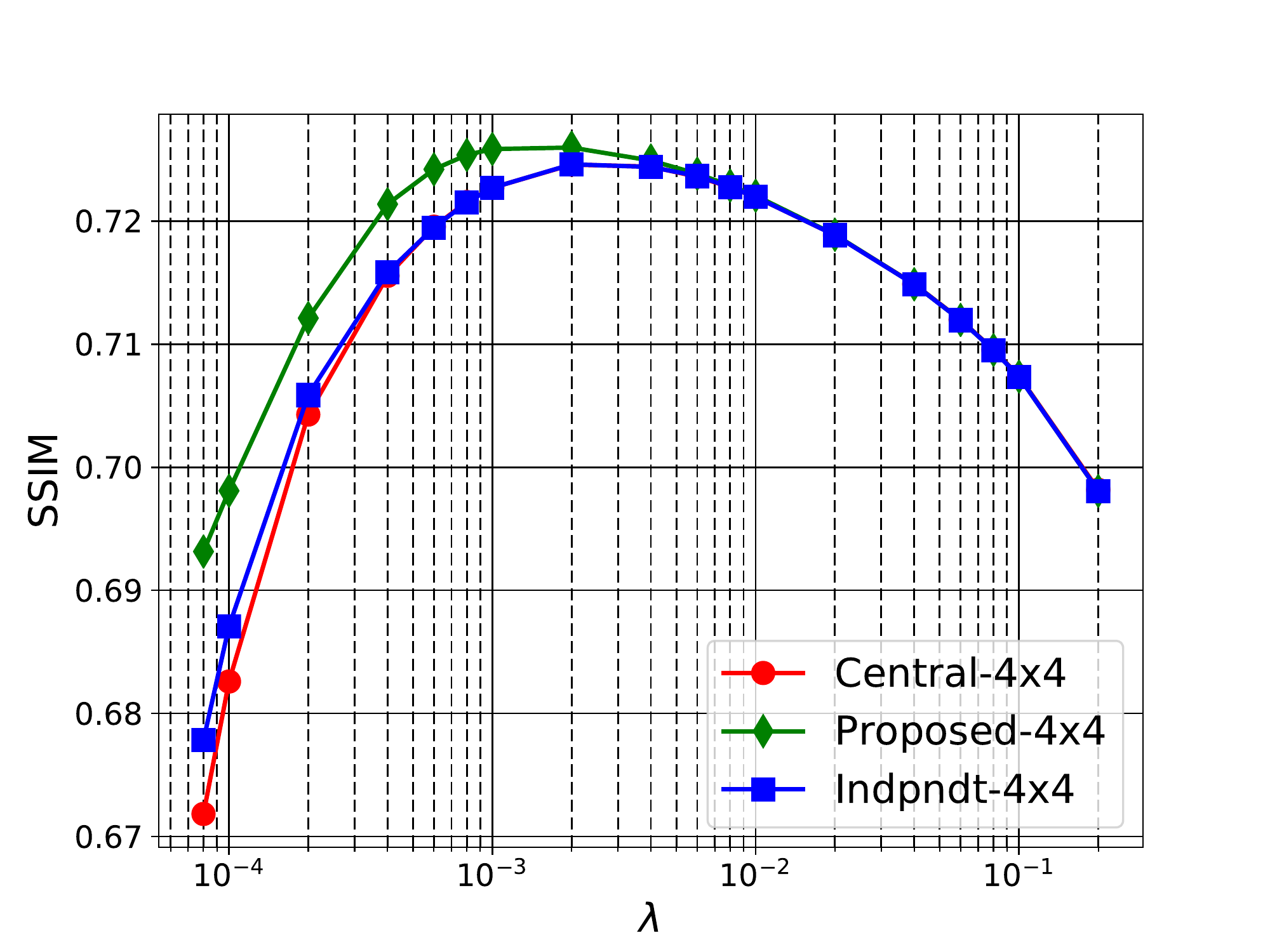}}
	\caption{Experiment 3: results from shift-invariant deblurring of ``Barbara'' image. The \emph{independent} and \emph{proposed} deblurring methods used $4 \times 4$ blocks with overlap of $100 \times 100$ pixels among them. The observed image had SNR = 8.7269 dB and SSIM = 0.6399.}
	\label{fig:expmnt3_shift_invariant_deblurring}
\end{figure*}

\begin{figure}
	\centering
	\subfloat[{SNR vs $\lambda$}]{\includegraphics[trim = 10mm 0mm 10mm 15mm, clip, width=0.98\linewidth]{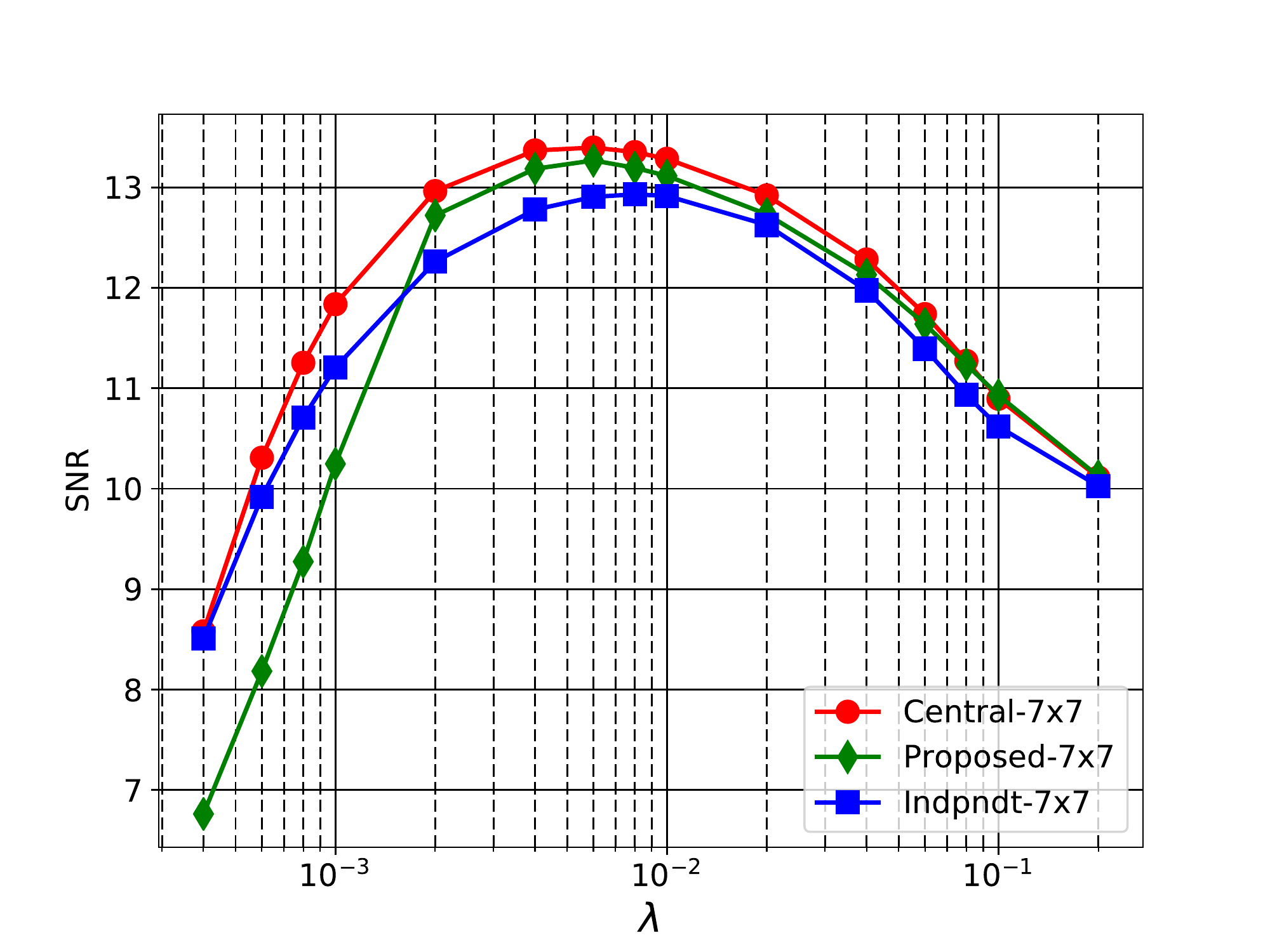}} \\
	\subfloat[{SSIM vs $\lambda$}]{\includegraphics[trim = 0mm 0mm 20mm 15mm, clip, width=0.98\linewidth]{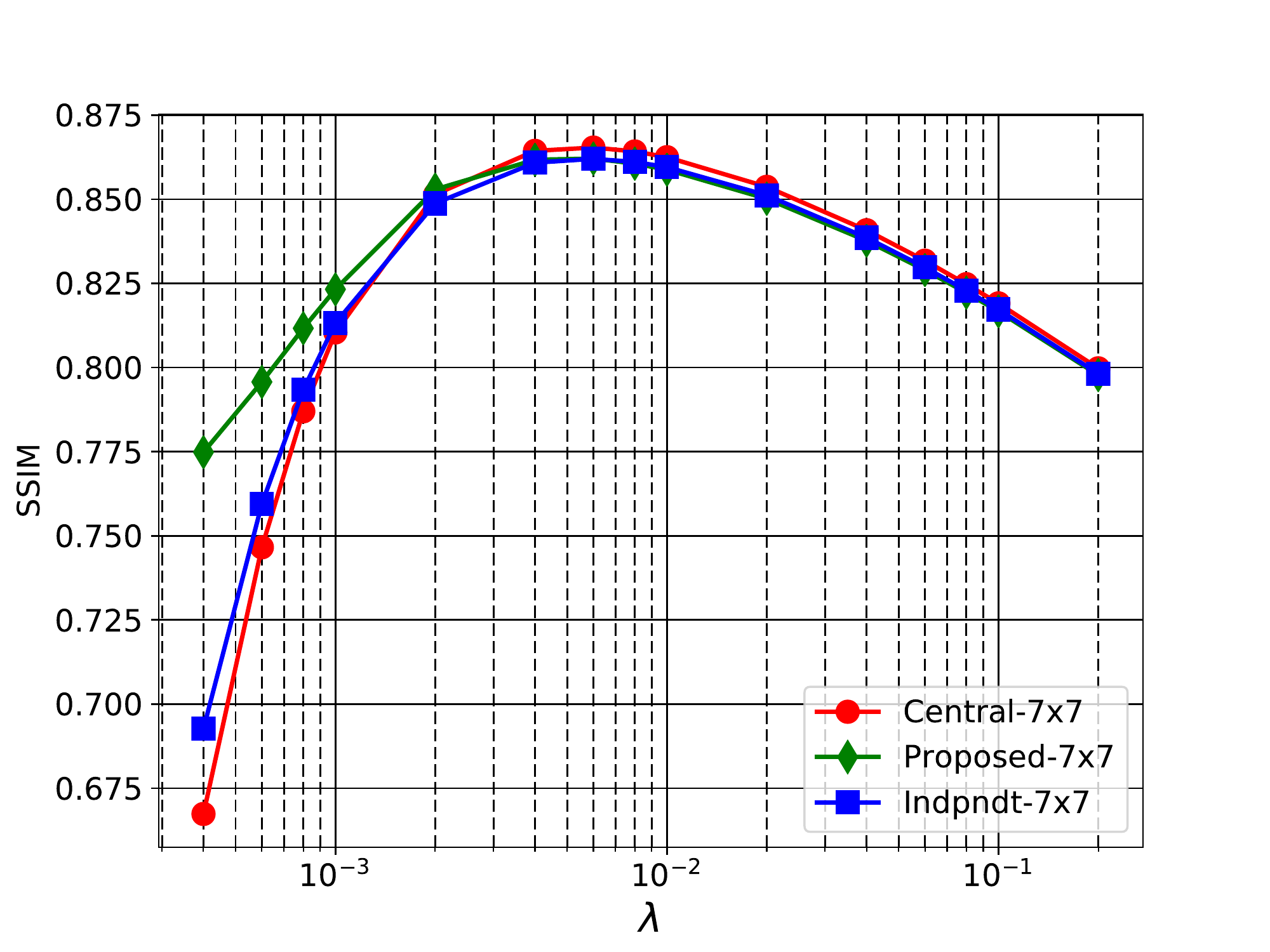}}
	\caption{Experiment 4: results from smooth shift-variant deblurring of ``Pentagon'' image shown below in Fig. \ref{fig:expmnt4_shift_variant_deblurring}. }
\end{figure}

\begin{figure*}
	\centering
	\subfloat[{Reference image (size = $2048 \times 2048$ pixels)}]{\includegraphics[trim = 0mm 0mm 0mm 0mm, clip, width=0.36\linewidth]{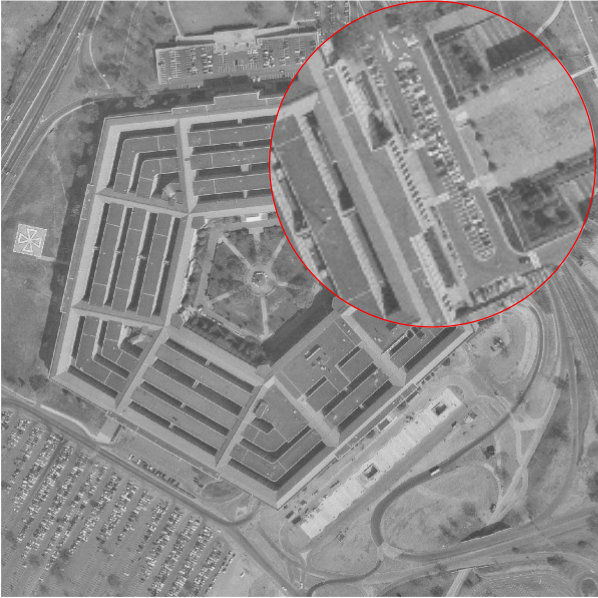}} \;
	\subfloat[{Observed image SNR = 8.71283 dB, SSIM = 0.7496}]{\includegraphics[trim = 0mm 0mm 0mm 0mm, clip, width=0.36\linewidth]{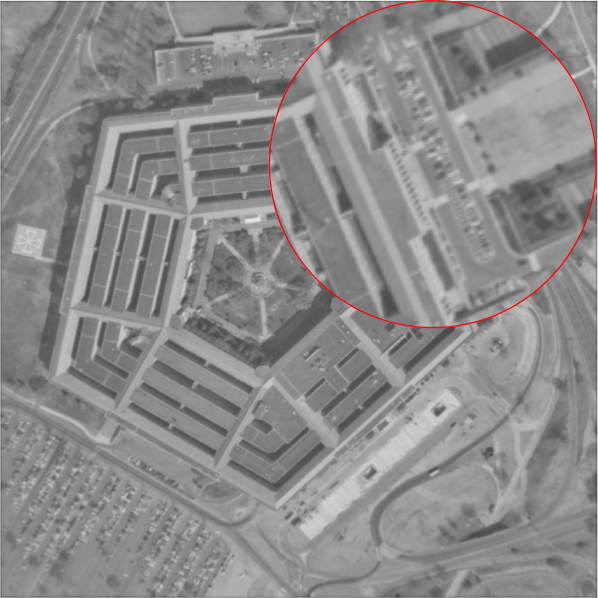}} \\
	\subfloat[{Estimated by \emph{centralized} deblurring (SNR = \textbf{13.3965} dB, SSIM = \textbf{0.8653})}]{\includegraphics[trim = 0mm 0mm 0mm 0mm, clip, width=0.36\linewidth]{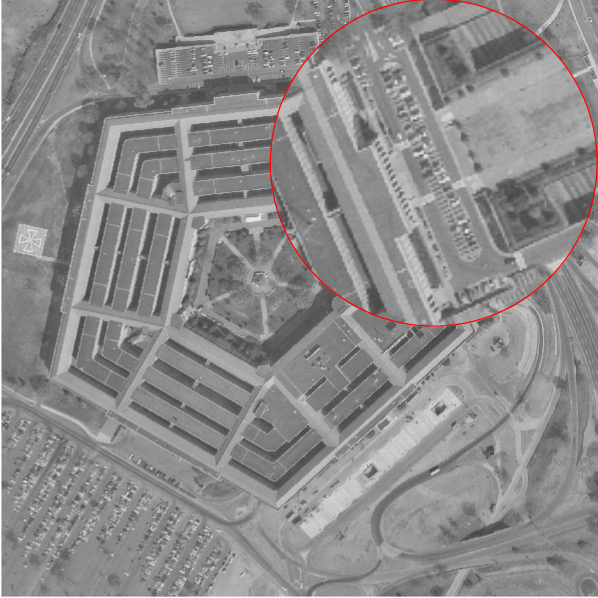}} \;
	\subfloat[{Estimated by \emph{independent} deblurring (SNR = 12.9053 dB, SSIM = 0.8619)}]{\includegraphics[trim = 0mm 0mm 0mm 0mm, clip, width=0.36\linewidth]{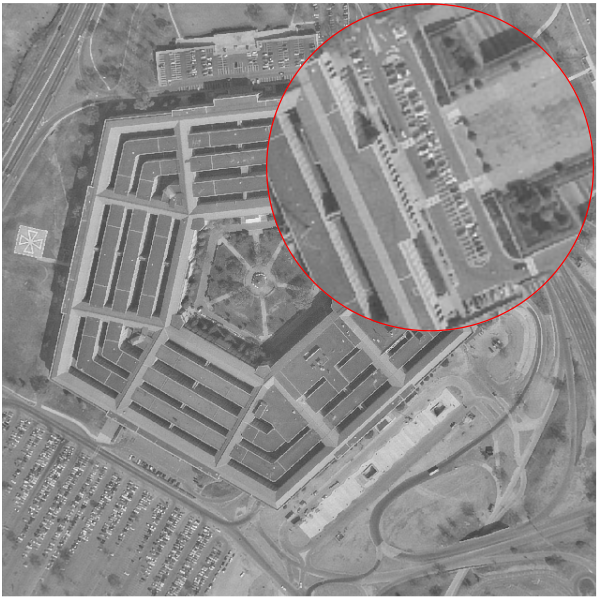}} \\
	\subfloat[{Estimated image blocks by \emph{proposed} deblurring}]{\includegraphics[trim = 0mm 0mm 0mm 0mm, clip, width=0.36\linewidth]{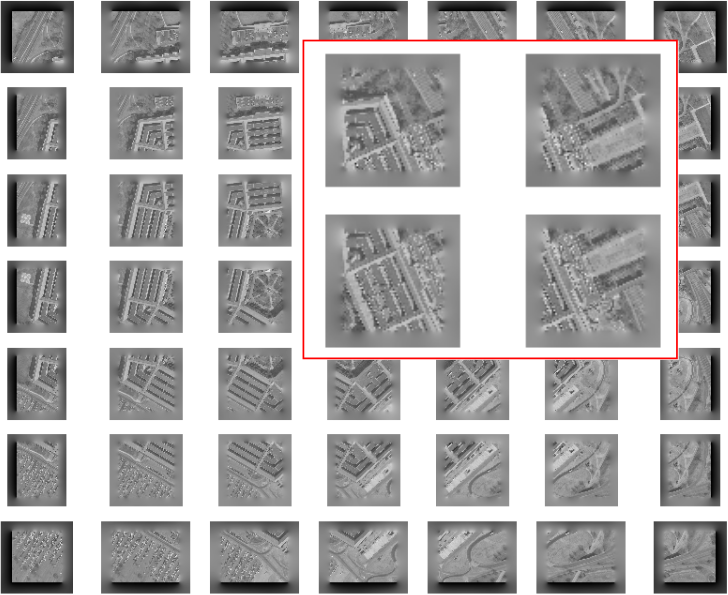}} \;
	\subfloat[{Estimated by \emph{proposed} deblurring (SNR = 13.2979 dB, SSIM = 0.8618)}]{\includegraphics[trim = 0mm 0mm 0mm 0mm, clip, width=0.36\linewidth]{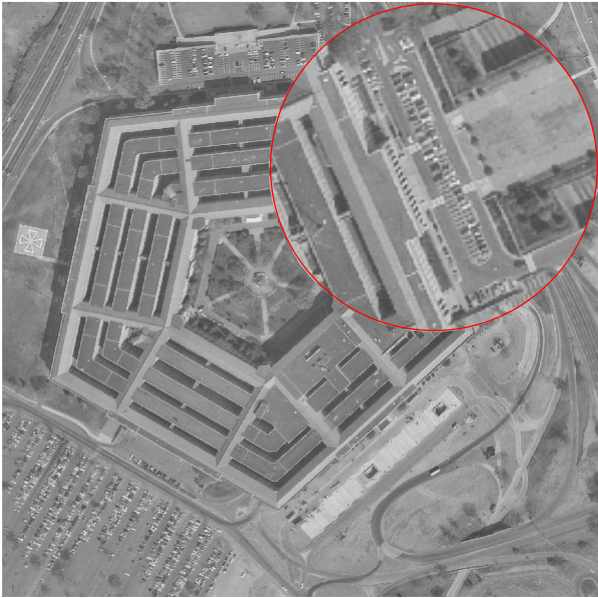}} \;
	\caption{Experiment 4: experimental setup and results from smooth shift-variant deblurring of ``Pentagon'' image. The observed image is generate by blurring with a $9 \times 9$ grid of  shift-variant normalized Gaussian PSFs having FWHM=$3.5 \times 3.5$ pixels in the center and linearly increasing FWHM in the radial direction up to $8.5 × 6.5$ pixels for the PSF at extreme corner. The deblurred image are estimated using only $7 \times 7$ grid of PSFs sampled within the field-of-view.}
	\label{fig:expmnt4_shift_variant_deblurring}
\end{figure*}

%
%
%
%

\end{document}